\documentclass[11pt]{amsart}

\usepackage[margin=1in]{geometry}

\usepackage{amsthm}
\usepackage{amssymb}
\usepackage{yhmath}
\usepackage{enumerate}
\usepackage{tikz-cd}
\usepackage{tkz-graph}
\usetikzlibrary{arrows}

\usepackage{mathrsfs}

\theoremstyle{plain}
\newtheorem{thm}{Theorem}[section]

\newtheorem{prop}[thm]{Proposition}
\newtheorem{lemma}[thm]{Lemma}
\newtheorem{cor}[thm]{Corollary}
\newtheorem*{heuristic}{Heuristic}

\newtheorem{thmx}{Theorem}
\newtheorem{thmy}{Theorem}
\newtheorem*{assumption}{Assumption}

\newtheorem*{question}{Question}

\theoremstyle{definition}  %Following theorems have plane text

\newtheorem{example}[thm]{Example}
\newtheorem{defn}[thm]{Definition}

\newcommand{\ran}{\mbox{\normalfont{ran}}}
\newcommand{\Span}{\mbox{\normalfont{span}}}
\newcommand{\Alg}{\mbox{\normalfont{Alg}}}
\newcommand{\ind}{\mbox{\normalfont{ind}}}

\newcommand{\C}{\mathbb{C}}
\newcommand{\A}{\mathfrak{A}}
\newcommand{\J}{\mathfrak{J}}
\newcommand{\BH}{\mathscr{B}(\mathscr{H})}
\newcommand{\HH}{\mathscr{H}}
\newcommand{\KK}{\mathscr{K}}

\newcommand{\KH}{\mathscr{K}(\mathscr{H})}

\newcommand{\Adw}{\mbox{Ad}_W}
\newcommand{\Adv}{\mbox{Ad}_V}
\newcommand{\Adu}{\mbox{Ad}_U}
\usepackage{xcolor}

\begin{document}

\title{Operator Algebras Generated by Left Invertibles}
\author{Derek DeSantis}

%%
%% For most people the defaults will be correct, so they are commented
%% out. To manually set these, just uncomment and make the needed
%% changes.
%% \college{Your college}
%% \city{Your City}
%%
%% For most people the following can be changed with a class
%% option. To manually set these, just uncomment the following and
%% make the needed changes.
%% \doctype{Thesis or Dissertation}
%% \degree{Your degree}
%% \degreeabbreviation{Your degree abbr.}
%%
%% Now that we know everything we need, we can generate the title page
%% itself.
%%
\maketitle
%%
%% You have a maximum of 350 words for your abstract, which includes
%% your title, name, etc.
%%
%% Required
\begin{abstract}
	 Operator algebras generated by partial isometries and their adjoints form the basis for some of the most well studied classes of C*-algebras. The primary object of this paper is the norm-closed operator algebra generated by a left invertible $T$ together with its Moore-Penrose inverse $T^\dagger$.   We denote this algebra by $\A_T$.  In the isometric case, $T^\dagger = T^*$ and $\A_T$ is a representation of the Toeplitz algebra. Of particular interest is the case when $T$ satisfies a non-degeneracy condition called analytic.  We show that $T$ is analytic if and only if $T^*$ is Cowen-Douglas.  When $T$ is analytic with Fredholm index $-1$, the algebra $\A_T$ contains the compact operators, and any two such algebras are boundedly isomorphic if and only if they are similar. 
	 
	 %Representations of such algebras encode the dynamics of orthonormal sets in a Hilbert space. We instigate a research program on concrete operator algebras that model the dynamics of Hilbert space frames. 
\end{abstract}

%% Optional
%% \begin{copyrightpage}
%% \end{copyrightpage}

%% Optional
%% \begin{dedication}
%% \end{dedication}

% Optional

%% Optional
%% \begin{grantinfo}
%% \end{grantinfo}
%% The ToC is required
%% Uncomment these if need be

%% The ToC is required
\tableofcontents
%% Uncomment these if need be
%\listoffigures
%\listoftables
%%
%% ``Real'' beginning of the document.
%% mainmatter is needed to set the page numbering correctly
%%   mainmatter is needed after the ToC, (LoF, and LoT) to set the
%%   page numbering correctly for the main body

%% Thesis goes here
%%%%%%%%%%%%%%%%%%%%%%%%%%%%%%%%%%%%%%%%
%%%%%%%%%%%%%%%%%%%%%%%%%%%%%%%%%%%%%%%%
%%%%%%%%%%%%%%%%%%%%%%%%%%%%%%%%%%%%%%%%
%%%%%%%%%%%%%%%%%%%%%%%%%%%%%%%%%%%%%%%%
%%%%%%%%%%%%%%%%%%%%%%%%%%%%%%%%%%%%%%%%
%%%%%%%%%%%%%%%%%%%%%%%%%%%%%%%%%%%%%%%%
%%%%%%%%%%%%%%%%%%%%%%%%%%%%%%%%%%%%%%%%
%%%%%%%%%%%%%%%%%%%%%%%%%%%%%%%%%%%%%%%%
\section{Introduction}
%%%%%%%%%%%%%%%%%%%%%%%%%%%%%%%%%%%%%%%%
%%%%%%%%%%%%%%%%%%%%%%%%%%%%%%%%%%%%%%%%
%%%%%%%%%%%%%%%%%%%%%%%%%%%%%%%%%%%%%%%%
%%%%%%%%%%%%%%%%%%%%%%%%%%%%%%%%%%%%%%%%
%%%%%%%%%%%%%%%%%%%%%%%%%%%%%%%%%%%%%%%%
%%%%%%%%%%%%%%%%%%%%%%%%%%%%%%%%%%%%%%%%
%%%%%%%%%%%%%%%%%%%%%%%%%%%%%%%%%%%%%%%%
%%%%%%%%%%%%%%%%%%%%%%%%%%%%%%%%%%%%%%%%

If $T \in \BH$ is any operator with closed range, then it follows from the open mapping theorem that there exists a unique operator $T^\dagger \in \BH$ such that
	\begin{enumerate}[i.]
		\item $\ker(T^\dagger) = \ran(T)^\perp = \ker(T^*)$
		\item $T^\dagger T x = x$ for each $x \in \ker(T)^\perp$.
	\end{enumerate}
The operator $T^\dagger$ that satisfies (i) and (ii) above is called the \textbf{Moore-Penrose Inverse of T}. The Moore-Penrose inverse behaves like a left inverse for $T$ on the domain where it makes sense. In particular if $T$ is left invertible, then $T^\dagger$ is the natural choice for a left inverse.  One can show that if $T$ is left invertible, then 
\[
T^\dagger = (T^* T)^{-1} T^*.
\]
Hence, if $T$ is isometric then $T^\dagger = T^*$.

Representations of operator algebras are often formed by choosing sufficiently nice linear maps on a Hilbert space that encode the features of some underlying algebraic object. Often, these maps are rigid in the sense that they will preserve Hilbert space structure from the domain into their range. For example, representations of C*-algebras are often formed by utilizing partial isometries, which preserve orthonormality of a set.  Left invertible operators, and more generally closed ranged operators, may not preserve orthonormality of sets.  However, they do preserve the looser Hilbert space structure of a \textit{frame}. This discussion lays the groundwork for a natural extension of C*-algebras of isometries, one that codifies frames over orthonormal bases. For further details, see \cite{DeSantis}.
 
% The previous discussion lays the groundwork for a natural extension of C*-algebras of isometries, one that codifies frames over orthonormal bases.  One arrives at a such an extension by replacing \textit{partial isometries and their adjoints} with \textit{closed range operators and their Moore-Penrose inverses}.  Therefore, by replacing all instances of ``unitary'' with ``invertible'', we arrive at a natural generalization of concrete C*-algebras - one that integrates dynamics of frame theory over orthonormal bases.  For further details, see \cite{DeSantis}.

Given a left invertible operator $T$ and its Moore-Penrose inverse $T^\dagger$, we define the operator algebra
\[
\A_T:= \overline{\mbox{\Alg}}(T,T^\dagger).
\]
where the closure is in the operator norm. If $T$ is an isometry, then its Moore-Penrose inverse $T^\dagger$ is $T^*$.  If $T$ is purely isometric (no unitary summand) with Fredholm index $-1$, then $T$ is unitarily equivalent to $T_z$ on $H^2(\mathbb{T})$.  Hence, $\A_T$ is the Toeplitz algebra $\mathcal{T}$. This representation is particularly nice, as every operator $A \in \mathcal{T}$ can be uniquely represented as a compact perturbation of a Toeplitz operator with continuous symbol.  The purpose of this paper is to understand the following question:

\begin{question}
    To what extent does $\A_T$ resemble the representations of C*-algebras generated by an isometry?
%	To what extent do the elements of $\A_T$ have the form ``compact perturbation of a continuous function''?
\end{question}

 The paper is organized as follows. In the second section, we review the background material needed for this paper.  This includes basic properties of left invertible operators required throughout the work, and elementary facts about $\A_T$ frequently used.  We discover that if the Fredholm index of $T$ is finite, $\A_T$ has the following description:

\begin{heuristic}
	If $T$ has finite Fredholm index, then the operators in $\A_T$ are compact perturbations of Laurent series.
\end{heuristic}
Our goal is to explore the extent to which this intuitive description is true. In that direction, we a canonical basis associated to $T$ is discovered.  We justify that in order to make any serious progress understanding the rich structure of $\A_T$, we need to restrict ourselves to a subclass of left invertible operators, known as analytic operators.   

In the third section, we discuss Cowen-Douglas operators, a class of operators that have rich analytic structure. In that section, we connect analyticity of $T$ to the class of Cowen-Douglas operators. Given an open set $\Omega \subset \C$ and a positive integer $n$, the operators in the Cowen-Douglas class $B_n(\Omega)$ are defined in Definition \ref{CDdef}. We prove the following connection:

\begin{thmx}
	Let $T \in \BH$ be  left invertible operator with Fredholm index equal to $-n$, for a positive integer $n \in \mathbb{N}$.  Then the following are equivalent:
	
	\begin{enumerate}[i.]
		
		\item $T$ is analytic
		
		\item ${T^\dagger}^*$ (the Cauchy Dual of $T$) is analytic
		
		\item There exists $\epsilon > 0$ such that $T^* \in B_n(\Omega)$ for $\Omega = \{z: |z|<\epsilon\}$
		
		\item There exists $\epsilon > 0$ such that $T^\dagger \in B_n(\Omega)$ for $\Omega = \{z: |z|<\epsilon\}$.
		
	\end{enumerate}

\end{thmx}
This result has several applications.  First, it gives an analytic model for representing $T$ in the sense that $T$ is unitarily equivalent to $M_z$ on a reproducing kernel Hilbert space of analytic functions. This further furnishes our description of $\A_T$ as ``compacts plus Laurent series''.  It also provides us with a decomposition theorem. If $T$ is an isometry, the Wold decomposition lets us decompose $T$ into a direct sum of Fredholm index $-1$ isometries (and a unitary).  A corollary of Theorem A is that we cannot reduce our study to the case where the Fredholm index of $T$ is $-1$. Rather, $T \sim \oplus T_j$ where each $T_j$ are strongly irreducible operators - operators that are analogous to Jordan blocks in $\BH$.

 Theorem A also allows us to analyze the isomorphism classes of $\A_T$ in the case when the Fredholm index of $T$ is $-1$.  This is carried out in Section Four.  Here, we determine the conditions for two such algebras to be isomorphic, establishing our main theorem.  It gives a rather rigid structure on bounded isomorphisms between the algebras $\A_T$:

\begin{thmx}	
	Let $T_i$, $i=1,2$ be left invertibles (analytic with Fredholm index $-1$) and $\A_i = \A_{T_i}$.  Suppose that  $\phi:\A_1 \rightarrow \A_2$ a bounded isomorphism.  Then there exists some invertible $V \in \BH$ such that $\phi(A) = V A V^{-1}$ for all $A \in \A_1$. 
\end{thmx}
In particular, this theorem shows that all bounded isomorphisms are completely bounded, and reduces the isomorphism problem to a similarity orbit problem.   We remark that the problem of finding the similarity orbit of Cowen-Douglas operators is classic. Using the results of Jiang and others on $K_0$ groups of strongly irreducible operators, we complete the classification in this case.

We conclude Section Four by investigating a class of illustrative examples arising from the theory of subnormal operators. If $S$ is a subnormal operator, we let $N = mne(S)$ denote the minimal normal extension of $S$, and $\sigma_{ap}(S)$ denote the approximate point spectrum of $S$. We show that in this class, $\A_S$  can be described by the heuristic of compact perturbations of Toeplitz operators with Laurent series:

\begin{thmx}
	Let $S$ be an analytic left invertible, Fredholm index $-1$, essentially normal, subnormal operator with $N:=mne(S)$ such that $\sigma(N) = \sigma_{ap}(S)$.  Let $\mathscr{B}$ be the uniform algebra generated by the functions $z$ and $z^{-1}$ on $\sigma_e(S)$. Then 
	\[
	\A_S = \{T_f + K: f \in \mathscr{B}, K \in \KH  \}.
	\]
	Moreover, the representation of each element as $T_f + K$ is unique. 
\end{thmx}

\section{Properties of Left Invertible Operators and $\A_T$	}
%%%%%%%%%%%%%%%%%%%%%%%%%%%%%%%%%%%%%%%%
%%%%%%%%%%%%%%%%%%%%%%%%%%%%%%%%%%%%%%%%
%%%%%%%%%%%%%%%%%%%%%%%%%%%%%%%%%%%%%%%%
%%%%%%%%%%%%%%%%%%%%%%%%%%%%%%%%%%%%%%%%
%%%%%%%%%%%%%%%%%%%%%%%%%%%%%%%%%%%%%%%%
%%%%%%%%%%%%%%%%%%%%%%%%%%%%%%%%%%%%%%%%
%%%%%%%%%%%%%%%%%%%%%%%%%%%%%%%%%%%%%%%%
%%%%%%%%%%%%%%%%%%%%%%%%%%%%%%%%%%%%%%%%

The focus of this section is elementary properties of left invertible operators and the algebra $\A_T$.  We will begin by discussing some basic facts about left invertible operators frequently used.  In order to make meaningful headway, we impose a Fredholm condition on our left invertibles.  We then discover some coarse properties of the algebra $\A_T$, noting that a dense set may be written as finite rank operator plus polynomials in $T$ and $T^\dagger$. This initiates our description of $\A_T$ as compact perturbations of Laurent series.   Drawing on analogies with isometric operators, we describe a non-degeneracy condition of left invertible operators called analytic. This allows one to build a type of basis on which $T$ acts like a shift operator. We conclude this section by demonstrating that one cannot hope to recover a decomposition exactly like the Wold decomposition for left invertible operators. 

%%%%%%%%%%%%%%%%%%%%%%%%%%%%%%%%%%%%%%%%
%%%%%%%%%%%%%%%%%%%%%%%%%%%%%%%%%%%%%%%%
%%%%%%%%%%%%%%%%%%%%%%%%%%%%%%%%%%%%%%%%
%%%%%%%%%%%%%%%%%%%%%%%%%%%%%%%%%%%%%%%%
\subsection{Basics of Left Invertible Operators}
%%%%%%%%%%%%%%%%%%%%%%%%%%%%%%%%%%%%%%%%
%%%%%%%%%%%%%%%%%%%%%%%%%%%%%%%%%%%%%%%%
%%%%%%%%%%%%%%%%%%%%%%%%%%%%%%%%%%%%%%%%
%%%%%%%%%%%%%%%%%%%%%%%%%%%%%%%%%%%%%%%%

For convenience in our discussion, we now list several well known facts about left invertable operators.  For proofs of these results, see \cite{DeSantis}.

\begin{prop}
	\label{left_facts}
%	\label{left inverses}
	Given any left invertible $T \in \BH$, the following hold:
	\begin{enumerate}[i.]
		\item $T T^\dagger$ is the (orthogonal) projection onto $\mbox{ran}(T)$
		\item $I-T T^\dagger$ is the (orthogonal) projection onto $\mbox{ran}(T)^\perp$
		\item $\ker(T^\dagger) = \ran(T)^\perp = \ker(T^*)$
		\item $\ran(T^\dagger) = \ran(T^*)$.
	\end{enumerate}
	Furthermore, every left inverse is of the form
	\[
	L = T^\dagger + A (I - T T^\dagger).
	\] 
	for some $A \in \BH$. If $S \in \BH$ satisfies $\|T - S\| < \|T^\dagger\|^{-1}$, then $S$ is also left invertible.  The operator $(T^\dagger S)^{-1} T^\dagger$ is a left inverse of $S$.  
\end{prop}

This paper will be concerned with the case when $\dim(\ran(T)^\perp) < \infty$.  This Fredholm assumption on $T$ will make the theory more interesting.  Furthermore, our interest is in left invertible operators which are not invertible.  We therefore make the following definition which summarizes these assumptions:

\begin{defn}
	An left invertible operator $T \in \BH$ is said to be \textbf{natural} if the $\dim(\ker(T^*))$ is a natural number.  Specifically,
\[
0<\dim(\ker(T^*)) = \dim(\ran(T)^\perp) < \infty
\]
\end{defn}
 Note that if $T$ is a natural left invertible, then $\ker(T^*)$ is a positive integer.  Hence, $T^*$ is not invertible, so neither is $T$. Moreover, we have the following Fredholm properties associated to natural left invertibles:

\begin{prop}
	\label{ess spec}
	Let $T$ be a natural left invertible.  Then $0 \in \sigma(T)$, and $0 \notin \sigma_e(T)$.  Indeed, $T$ is Fredholm with $\mbox{\ind}(T) = - \mbox{dim}(\ker(T^\dagger)) = -\mbox{\ind}(T^\dagger)$.
\end{prop}

\begin{proof}
	Since $T$ is not invertible, $0 \in \sigma(T)$. As  $\dim(\ran(T)^\perp) < \infty$, and $I- T T^\dagger$ is the projection onto the $\ran(T)^\perp$, $T$ is invertible in $\BH / \KH$.  Therefore, $T$ is Fredholm.  Because $T$ is injective, the Fredholm index of $T$ is 
	\[
	\mbox{\ind}(T) = \mbox{dim}(\ker(T)) - \mbox{dim}(\ker(T^*)) = -\mbox{dim}(\ker(T^\dagger)).
	\]
	Note that $(T^\dagger)^* = T(T^* T)^{-1}$.  Hence $(T^\dagger)^*$ is injective, so that 
	\[
	\mbox{\ind}(T^\dagger) = \mbox{dim}(\ker(T^\dagger)) - \mbox{dim}(\ker((T^\dagger)^*)) = \mbox{dim}(\ker(T^\dagger)) = - \mbox{\ind}(T).
	\]
\end{proof}

\begin{cor}
	If $T$ is a natural left invertible, then all left inverses $L$ of $T$ are finite rank perturbations of $T^\dagger$.  Hence, all left inverses $L$ of $T$ are Fredholm with index $\mbox{\ind}(L) = \mbox{dim}(\ker(T^\dagger)) = \mbox{\ind}(T^\dagger)$.
\end{cor}

\begin{prop}
	\label{close ind}
	If $T, S \in \BH$, $T$ is a natural left invertible and $\|T - S\| < \|T^\dagger\|^{-1}$, then $S$ is Fredholm with $\mbox{\ind}(S) = \mbox{\ind}(T)$.
\end{prop}

\begin{proof}
	Let $\tilde{T}:\HH \rightarrow \ran(T)$ be the restriction of $T$.  Then $\tilde{T}$ is invertible, with $\|\tilde{T}\| = \|T^\dagger\|$.  Therefore, if $A \in \mathscr{B}(\HH, \ran(T))$ with 
	\[
	\|A - \tilde{T}\| < \|\tilde{T}\|^{-1} = \|T^\dagger\|^{-1}
	\]
	then $A$ is invertible as well.  
	
	By assumption, $T$ has closed range.  So, $\HH = \ran(T) \oplus \ran(T)^\perp$.  Write $S = S_1 + S_2$ where $S_1 = P_{\ran(T)} S$ and $S_2 = P_{\ran(T)^\perp} S$.  Then,
	\[
	\|S_1 - \tilde{T}\| = \|P_{\ran(T)}(S - T)\| < \|T^\dagger\|^{-1}.
	\]
	Hence, $S_1$ is invertible.  Moreover since $\dim(\ran(T)^\perp)$ is finite, $S_2 \in \KH$.  Therefore, $S$ is a compact perturbation of an invertible operator, and thus Fredholm. 
	
	By Lemma \ref{left_facts}, $S$ is left invertible with left inverse $L = (T^\dagger S)^{-1} T^\dagger$. By Proposition \ref{left_facts}, $S^\dagger = K + L$ for some compact $K \in \KH$. Therefore, 
	\[
	\mbox{\ind}(S^\dagger) = \mbox{\ind}((T^\dagger S)^{-1} T^\dagger) = \mbox{\ind}((T^\dagger S)^{-1}) = \mbox{\ind}( T^\dagger) = \mbox{\ind}( T^\dagger)
	\]
	By Proposition \ref{ess spec}, $\mbox{\ind}(S) = \mbox{\ind}(T)$.
\end{proof}
\begin{defn}
Given a natural left invertible $T \in \BH$, we reserve the notation
\[
\mathscr{E}:=  \ran(T)^\perp = \ker(T^\dagger) = \ker(T^*).
\]
\end{defn}
For isometric operators, $T^n \mathscr{E} \perp T^m \mathscr{E}$ for all $n \neq m$.  This is not true for general left invertible operators, even though $\mathscr{E}$ is perpendicular to the range of $T$.  However, it is true that $\ker((T^\dagger)^n) = \bigvee_{k=0}^{n-1} T^k \mathscr{E}$:

\begin{prop}
	\label{ker dagger n}
	Let $T$ be a natural left invertible, and $P = I - T T^\dagger$ be the projection onto $\mathscr{E}$. Then for each $n \geq 1$, we have
	\begin{equation}
	\label{poly}
	I - T^n {T^\dagger}^n = \sum_{k=0}^{n-1} T^k P {T^\dagger}^k.
	\end{equation}
	Consequently, 
	\[
	\ker((T^\dagger)^n) = \bigvee_{k=0}^{n-1} T^k \mathscr{E}.
	\]
\end{prop}

\begin{proof}
	By a telescopic sum, $I - T^n {T^\dagger}^n = \sum_{k=0}^{n-1} T^k P {T^\dagger}^k$.  To prove the set equality, suppose $x \in \bigvee_{k=0}^{n-1} T^k \mathscr{E}$.  Then it follows immediately that ${T^\dagger}^n x = 0$.  On the other hand, if $x \in \ker((T^\dagger)^n)$, then by Equation (\ref{poly}),
	\[
	x = (I - T^n {T^{\dagger}}^n)x = \sum_{k=0}^{n-1} T^k P {T^\dagger}^k x.
	\]
	Since $P {T^\dagger}^k x \in \mathscr{E}$ for all $k$, it follows that $x \in \bigvee_{k=0}^{n-1} T^k \mathscr{E}$. 
\end{proof}

%%%%%%%%%%%%%%%%%%%%%%%%%%%%%%%%%%%%%%%%
%%%%%%%%%%%%%%%%%%%%%%%%%%%%%%%%%%%%%%%%
%%%%%%%%%%%%%%%%%%%%%%%%%%%%%%%%%%%%%%%%
%%%%%%%%%%%%%%%%%%%%%%%%%%%%%%%%%%%%%%%%
\subsection{Basic Properties of  $\A_T$}
%%%%%%%%%%%%%%%%%%%%%%%%%%%%%%%%%%%%%%%%
%%%%%%%%%%%%%%%%%%%%%%%%%%%%%%%%%%%%%%%%
%%%%%%%%%%%%%%%%%%%%%%%%%%%%%%%%%%%%%%%%
%%%%%%%%%%%%%%%%%%%%%%%%%%%%%%%%%%%%%%%%

We now analyze the basics of the algebra $\A_T$. If $S$ is an isometry, it dilates to a unitary.  Moreover, if $\mathscr{C}$ is the commutator ideal of $C^*(S)$ and $\pi: C^*(S) \rightarrow C^*(S)/\mathscr{C}$, then $\pi(S)$ is a unitary. When the Fredholm index of $S$ is $-1$, $\mathscr{C} = \KH$, the compact operators.  Consequently, one obtains that the Toeplitz algebra is compact perturbations of continuous functions on the unit circle. 

Analogous statements are also true for general left invertibles.  If $T$ is left invertible, it dialtes to an invertible.  If $\mathscr{C}$ is the commutator ideal of $\A_T$, then and $\pi(T)$ is invertible in $\A_T / \mathscr{C}$.  As a consequence, we use the following heurstic to describe $\A_T$:

\begin{heuristic}
	The algebra $\A_T$ is compact perturbations of Laurent series centered at zero.  
\end{heuristic}

We prove that when the dimension of $\ker(T^*)$ is finite, $\mathscr{C} \subset \KH$. We then show that $\A_T / \mathscr{C}$ consists of formal Laurent polynomials, namely polynomials in $z$ and $z^{-1}$. Moreover $T$ may also be dilated to an invertible, allowing us to identify $\A_T$ as the corner of the algebra generated by this invertible.  Combining these results allows one to parameterize the commutator ideal.

It is important to note that the projection $P = I - T T^\dagger = T^\dagger T - T T^\dagger \in \mathscr{C}$.  We begin this section with some simple observations that will be used throughout the paper:

\begin{lemma}
	Let $T$ be a left invertible operator.  Then $\A_T \subset C^*(T)$.   If $T$ is natural, then $\mathscr{C} \subset \KH$.
\end{lemma}

\begin{proof}
	Since $T^\dagger = (T^*T)^{-1} T^*$, $T^\dagger \in C^*(T)$. Therefore, $\A_T \subset C^*(T)$. Now suppose that $T$ is natural. Let $X = T^n {T^\dagger}^m$ and $Y = T^k {T^\dagger}^l$.  If we can show that $XY - YX$ is finite rank, then it will follow from taking linear combinations and limits that  $\mathscr{C} \subset \KH$.  
	
	To this end, notice that 
	\[
	XY - YX = T^n {T^\dagger}^m T^k {T^\dagger}^l -  T^k {T^\dagger}^l T^n {T^\dagger}^m
	\]
	Now if $m \leq k$,  $ T^n {T^\dagger}^m T^k {T^\dagger}^l = T^{n+k-m} {T^\dagger}^l$.  On the other hand, if $m \geq k$, then $T^n {T^\dagger}^m T^k {T^\dagger}^l = T^n {T^\dagger}^{l+m-k}$. Likewise, $T^k {T^\dagger}^l T^n {T^\dagger}^m = T^{n+k-l} {T^\dagger}^m$ if $l \leq n$ and $T^k {T^\dagger}^{l+m-n}$ otherwise.  Therefore, the expression $T^n {T^\dagger}^m T^k {T^\dagger}^l -  T^k {T^\dagger}^l T^n {T^\dagger}^m$ can be simplified depending on the values of $n,m,k$ and $l$.  This leaves us with eight total cases to check.  For example, two cases arise from $m \geq k$ and $l \geq n$.  By above, if $m \geq k$ and $l \geq n$, then
	\[
	T^n {T^\dagger}^m T^k {T^\dagger}^l - T^k {T^\dagger}^l T^n {T^\dagger}^m = T^n {T^\dagger}^{l+m-k} - T^k {T^\dagger}^{l+m-n}.
	\]
	This leaves us with two sub-cases: either $n \leq k$ or $k \leq n$.  If $n \leq k$, we have 
	\[
	T^n {T^\dagger}^{l+m-k} - T^k {T^\dagger}^{l+m-n} = T^n ( I - T^{k-n} {T^\dagger}^{k-n}) {T^\dagger}^{l + m - k}.
	\]
	By Proposition \ref{ker dagger n}, $ I - T^{k-n} {T^\dagger}^{k-n}$ is a sum of finite rank operators, and thus, $T^n {T^\dagger}^m T^k {T^\dagger}^l -  T^k {T^\dagger}^l T^n {T^\dagger}^m$ is finite rank.  The case when $k \leq n$ is the same.  The other six cases are similar.  
\end{proof}

We now investigate the quotient of $\A_T$ by the commutator ideal $\mathscr{C}$.  Let $\pi$ denote the canonical map $\pi: \A_T \rightarrow \A_T / \mathscr{C}$. As $P = I - T T^\dagger$ is  in $ \mathscr{C},$ it follows that $\pi(T)$ is invertible with inverse $\pi(T^\dagger)$.  Hence, $\A_T / \mathscr{C}$ is a commutative Banach algebra generated by the invertible $\pi(T)$ and  its inverse $\pi(T^\dagger)$.  

Note that if $\A$ is a commutative unital Banach algebra generated by an invertible $a$ and its inverse $a^{-1}$, then the character space $\Omega(\A)$ is homeomorphic to $\sigma(a)$. Hence, the Gelfand map provides a norm decreasing homomorphism of
\[
\Gamma: \A_T / \mathscr{C} \rightarrow C(\sigma(\pi(T))). 
\]
For each $\lambda \in \sigma(\pi(T))$, let $z: \sigma(\pi(T)) \hookrightarrow \C$ represent the inclusion function. Namely, $z(\lambda) = \lambda$ for all $\lambda \in \sigma(\pi(T))$.  Then $z$ is invertible by construction, with inverse $z^{-1}(\lambda):=\lambda^{-1}$ for all $\lambda \in \sigma(\pi(T))$. Under the Gelfand identification, $\pi(T) \mapsto z$ and $\pi(T^\dagger) \mapsto z^{-1}$ on $\sigma(\pi(T))$. Consequently, $z$ and $z^{-1}$ generate the image of $\A_T / \mathscr{C}$ under $\Gamma$.   In this sense, $\A_T / \mathscr{C}$ consists of Laurent polynomials centered at zero.

A few comments are necessary at this point. First, the Gelfand map need not have closed range, and thus, $\Gamma(\A_T / \mathscr{C})$ may not be complete.  Moreover, $\Gamma$ may not even be injective in general.  If $\A$ is a commutative Banach algebra, and $a \in \A$ has $\sigma(a) = 0$, then $\Gamma(a) = 0$. However, since $\A_T / \mathscr{C}$ is generated by $\pi(T)$ and $\pi(T^\dagger) = \pi(T)^{-1}$, it follows that $z$ (and therefore $z^{-1}$) are non-zero. As $\Gamma$ is norm decreasing, we do have that every function in the range of $\Gamma$ is a Laurent series in $z$ and $z^{-1}$.

It will be shown in Section 4.1 that when the Fredholm index of $T$ is $-1$, $\mathscr{C} = \KH$. In some cases, this furnishes a rather detailed analysis of the quotient. In particular, the case of essentially normal subnormal operators will be studied in Section 4.4.  However, presently we will concern ourselves with an algebraic characterization of the commutator ideal.  To do this, we will first get a description of the algebra generated by $T$ and $T^\dagger$ pre-closure. 

% Consider the following example:
%
%\begin{example}
%	
%	Suppose that $T$ were a left invertible, irreducible, essentially normal operator and $\pi:\BH \rightarrow \BH/\KH$.  Suppose further that $\mathscr{J} = \KH$. By Proposition \ref{ess spec}, $\pi(T)$ is invertible, with inverse $\pi(T^\dagger)$.  Since $\A_T \subset C^*(T)$, we must have that in the Calkin algebra $\pi(\A_T) \subset \pi(C^*(T))$.  Since $T$ is irreducible, and since $P = I-T T^\dagger$ is a compact operator inside $C^*(T)$, necessarily $C^*(T)$ must contain all the compacts. Since $T$ is essentially normal, we can apply the Gelfand map to get
%	\[
%	\pi(C^*(T)) = C^*(T) / \KH \cong C(\sigma_e(T)).
%	\]
%	Consequently, $\pi(\A_T) \cong \overline{\mbox{Alg}}(z, z^{-1})$ as functions on $\sigma_e(T)$.  That is, $\pi(\A_T)$ is the uniform sub-algebra of $C(\sigma_e(T))$ generate by Laurent series centered at zero.   We will return to this example in a later section. 
%\end{example}

We just analyzed how quotienting by the commutator ideal results in $T$ becoming invertible. As a consequence, ``Laurent polynomials'' in $z$ and $z^{-1}$ over $\sigma(\pi(T))$ are dense in the quotient.  Next, we observe that if $T \in \BH$ is left invertible, then it dilates to an invertible. This will allow us to succinctly describe $\mbox{\Alg}(T,T^\dagger)$.  

Let $P = I - T T^\dagger$. Then the operator $W \in \mathscr{B}(\HH \oplus \HH)$ given by 
\[
W = 
\overset{\HH \ \ \HH}{
\begin{pmatrix}
T^\dagger & 0\\
P & T
\end{pmatrix}
}
\]
is invertible, with inverse given by 
\[
W^{-1} = \begin{pmatrix}
T & P\\
0 & T^\dagger
\end{pmatrix}.
\]
Let $Q_1$ and $Q_2$ denote the projections onto $\HH_1:=\HH \oplus 0$ and $\HH_2:=0 \oplus \HH$ respectively. By construction $T = Q_2 W \mid_{\HH_2}$ and $T^\dagger = Q_2 W^{-1} \mid_{\HH_2}$. Furthermore, for each $n$,
\[
W^n = \begin{pmatrix}
{T^\dagger}^n & 0\\
D_n & T^n
\end{pmatrix}
\quad \quad \quad
W^{-n} = \begin{pmatrix}
T^n & D_n\\
0 & {T^\dagger}^n
\end{pmatrix}
\]
where $D_n:= \sum_{k=0}^{n-1} T^k P {T^\dagger}^{n-1-k}$.  Since $\mbox{dim}(\mathscr{E}) < \infty$ by assumption, $D_n$ is a finite rank operator for each $n$. Furthermore, for every $n$,  $T^n = Q_2 W^n \mid_{\HH_2}$ and ${T^\dagger}^n = Q_2 W^{-n} \mid_{\HH_2}$.  It therefore follows that $\mbox{\Alg}(T,T^\dagger) = Q_2 \mbox{\Alg}(W \mid_{\HH_2}, W^{-1} \mid_{\HH_2})$.  Now, a straightforward calculation reveals the following:

\begin{equation}
\label{parts}
\begin{array}{rcl}
\vspace{.1 in}
Q_2 W^{-n} Q_1 W^m \mid_{\HH_2} &=& 0\\
\vspace{.1 in}
Q_2 W^{-n} Q_2 W^m \mid_{\HH_2} &=& {T^\dagger}^n T^m\\
\vspace{.1 in}
Q_2 W^m Q_1 W^{-n} \mid_{\HH_2} &=& D_m D_n\\
Q_2 W^m Q_2 W^{-n} \mid_{\HH_2} &=& T^m {T^\dagger}^n.
\end{array}
\end{equation}
Since $\mbox{\Alg}(T,T^\dagger) = Q_2 \mbox{\Alg}(W \mid_{\HH_2}, W^{-1} \mid_{\HH_2})$,  the operators appearing in Equation (\ref{parts}) span $\mbox{\Alg}(T, T^\dagger)$.  Namely, using Equation (\ref{parts}) we have
\begin{equation}
\label{part1}
D_m D_n + T^m {T^\dagger}^n = Q_2 W^{m} W^{-n} \mid_{\HH_2} = Q_2 W^{m-n} \mid_{\HH_2} = \left\{
\begin{array}{cl}
T^{m-n} & \mbox{ if } m > n\\
{T^\dagger}^{m-n} & \mbox{ else.}
\end{array}
\right. 
\end{equation}
Also,
\begin{equation}
\label{part2}
{T^\dagger}^n T^m = Q_2 W^{-n} W^m \mid_{\HH_2} = \left\{
\begin{array}{cl}
T^{m-n} & \mbox{ if } m > n\\
{T^\dagger}^{n-m} & \mbox{ else.}
\end{array}
\right. 
\end{equation}
Thus,  $T^m {T^\dagger}^n$ is equal to some power of a generator, up to the finite rank perturbation $D_m D_n$. 
Consequently, every operator $A$ in $\mbox{\Alg}(T, T^\dagger)$ may be written in the form 
\[
F + \sum_{k=0}^N a_k T^k + \sum_{l=1}^M b_l {T^\dagger}^l,
\]
where $F$ is some finite rank operator.  Hence, the dense subalgebra $\mbox{\Alg}(T, T^\dagger)$ are finite rank operators plus Laurent polynomials in $T$ and $T^\dagger$. We record this result here for future reference:

\begin{prop}
	\label{alg simple}
	Let $T$ be a natural left invertible operator.  If $A \in \mbox{\Alg}(T, T^\dagger)$ (pre-closure of $\A_T$), is the operator 
	\[
	A = \sum_{n,m=0}^N \alpha_{n,m} T^m {T^\dagger}^n
	\]
	then $A$ may be rewritten as
	\[
	A = F + \sum_{N \geq m \geq n \geq 0} \alpha_{n,m} T^{m-n} + \sum_{N \geq n \geq m \geq 1} \alpha_{n,m} {T^\dagger}^{n-m}
	\]
	where $F$ is the finite rank operator given by $F = -\sum_{n,m=0}^N \alpha_{n,m} D_m D_n$, and  $D_n= \sum_{k=0}^{n-1} T^k P {T^\dagger}^{n-1-k}$.
\end{prop}

We conclude this subsection with a final comment on the commutator ideal $\mathscr{C}$ of $\A_T$. Recall that $P = I - T T^\dagger \in \mathscr{C}$.  Hence by the preceding, all the finite rank operators $F$ from this construction are in the commutator ideal $\mathscr{C}$. Combined with Proposition \ref{alg simple}, this observation allows us to algebraically characterize a dense subset of $\mathscr{C}$.

\begin{prop}
	\label{commutator}
	Let  $P = I - T T^\dagger$ and set 
	\[
	\mathscr{K}_T:= \overline{\mbox{\Span}} \{T^n P {T^\dagger}^m: n,m \geq 0\}.
	\]
	Then $\mathscr{K}_T = \mathscr{C}$. 
\end{prop}

\begin{proof}
	First we show that $\mathscr{K}_T$ is an ideal of $\A_T$.  If $A \in \mbox{\Alg}(T, T^\dagger)$, then by Proposition \ref{alg simple}, 
	\[
	A = -\sum_{n,m=0}^N\alpha_{n,m} D_m D_n + \sum_{N \geq m \geq n \geq 0} \alpha_{n,m} T^{m-n} + \sum_{N \geq n \geq m \geq 1} \alpha_{n,m} {T^\dagger}^{n-m}.
	\]
	Now consider the product $A (T^k P {T^\dagger}^l)$ for some $k,l$.  Using Equations (\ref{part1}) and (\ref{part2}), it follows that $ T^k P {T^\dagger}^l$ multiplied by any part  in the decomposition of $A$ above is once again in $\mbox{\Span} \{T^n P {T^\dagger}^m: n,m \geq 0\}$. Similarly,  $ (T^k P {T^\dagger}^l)A \in \mbox{\Span} \{T^n P {T^\dagger}^m: n,m \geq 0\}$.  It follows that all polynomials from $\mbox{\Span} \{T^n P {T^\dagger}^m: n,m \geq 0\}$ multiplied by $A$ belong to $\mbox{\Span} \{T^n P {T^\dagger}^m: n,m \geq 0\}$. If $B \in \mathscr{K}_T$, it follows from taking limits and using the closure of $\mathscr{K}_T$ that $AB, BA \in \mathscr{K}_T$.  By density of $\mbox{\Alg}(T, T^\dagger)$ in $\A_T$, we have that $\mathscr{K}_T$ is an ideal for $\A_T$. 
	
	By definition, $P \in \mathscr{K}_T$ and so, $\A_T / \mathscr{K}_T$ is commutative.  Hence, $\mathscr{C} \subseteq \mathscr{K}_T$. However, notice that $\mathscr{K}_T$ is the principal ideal generated by $P$.  Indeed, if $\mathscr{J}$ is an ideal of $\A_T$, and $P \in \mathscr{J}$, then at a minimum each $T^n P {T^\dagger}^m$ must be inside of $\mathscr{J}$. Hence, $\mathscr{K}_T = \mathscr{C}$.
\end{proof}

%\rmkdd{Remove?} Ideally, we would like a canonical representation of $T$ as multiplication by $z$ on some reproducing kernel Hilbert space. If we further have $T^\dagger$ represented as multiplication by $z^{-1}$, then $\A_T$ could be further described as compact perturbations of multiplication operators with symbols Laurent series. This turns out to be the case for special class of operators, which we call analytic.  We will expand on this particular topic in our discussion of Cowen-Douglas operators. 

%%%%%%%%%%%%%%%%%%%%%%%%%%%%%%%%%%%%%%%%
%%%%%%%%%%%%%%%%%%%%%%%%%%%%%%%%%%%%%%%%
%%%%%%%%%%%%%%%%%%%%%%%%%%%%%%%%%%%%%%%%
%%%%%%%%%%%%%%%%%%%%%%%%%%%%%%%%%%%%%%%%
\subsection{Wold-Type Decompositions}
%%%%%%%%%%%%%%%%%%%%%%%%%%%%%%%%%%%%%%%%
%%%%%%%%%%%%%%%%%%%%%%%%%%%%%%%%%%%%%%%%
%%%%%%%%%%%%%%%%%%%%%%%%%%%%%%%%%%%%%%%%
%%%%%%%%%%%%%%%%%%%%%%%%%%%%%%%%%%%%%%%%

  One reason isometries are a  tractable class of operators is due to the celebrated Wold decomposition.  For future notational considerations, we state the Wold Decomposition here:

\begin{thm}[Wold Decomposition for Isometries]
	Let $S$ be an isometry on $\HH$.  Define
	\[
	\begin{array}{rcl}
	\HH_I &:=& \bigcap_{n \geq 1} S^n \HH\\
	\HH_A &:=&  \bigvee_{n \geq 0} S^n \mathscr{E}.
	\end{array}
	\]
	Then $\HH_I$ and $\HH_A$ are reducing for $S$, $\HH = \HH_I \oplus \HH_A$,  $S \mid_{\HH_I}$ is a unitary and $S \mid_{\HH_A}$ is a unilateral shift of rank $n$.
\end{thm}

All isometries decompose the Hilbert space into two orthogonal,  reducing subspaces for $S$.  On $\HH_I$, the isometry $S$ is invertible, and hence, a unitary.  On $\HH_A$, the isometry is purely isometric.  The isometric summand yields an analytic model.  Concretely, $S\mid_{\HH_A}$ is unitarily equivalent to $\mbox{dim}(\mathscr{E})$ orthogonal copies of $M_z$ on $H^2(\mathbb{T})$.  The unilateral shift is is unitarily equivalent to the operator of multiplication by $z$ on a reproducing kernel Hilbert space of analytic functions.

For a general left invertible operator $T \in \BH$, one would like to arrive at a similar type of decomposition.  We make the following definition:
\begin{defn}
	Given a left invertible $T \in \BH$, we  define:
	
	\[
	\begin{array}{rcl}
	\HH_I &:=& \bigcap_{n \geq 1} T^n \HH\\
	\HH_A &:=&  \bigvee_{n \geq 0} T^n \mathscr{E}.
	\end{array}
	\]
\end{defn}
As a caution to the reader, $\HH_I$ and $\HH_A$ need not be reducing.  However, $\HH_I$ and $\HH_A$ are clearly invariant subspaces for $T$.  Moreover, $\HH_I$ is invariant for $T^\dagger$ and  $T \mid_{\HH_I}$ is invertible, with inverse $T^\dagger \mid_{\HH_I}$.  We shall show that $T\mid_{\HH_A}$ acts like a shift, not on a orthonormal basis, but on a more general basis. This will be discussed below.

In general, there is no Wold-type decomposition for $T$ with regards to the spaces $\HH_I$ and $\HH_A$.  Namely, it is not the case that $\HH= \HH_I \oplus \HH_A$.  Their sum can fail to be orthogonal, and hence, $\HH_I + \HH_A$ may not be equal to $\HH$.  However, it is always the case that if $T \in \BH$ is left invertible, then $\HH_I + \HH_A$ is dense in $\HH$ with $\HH_I \cap \HH_A = 0$. See \cite{DeSantis} for details.

In \cite{Shimorin}, Shimorin observed that there is almost a Wold-type decomposition.  This decomposition is related to a canonical left invertible operator associated to $T$, called the Cauchy dual of $T$: 

\begin{defn}[\cite{Shimorin}]
	Given a left invertible operator $T$, the \textbf{Cauchy dual} of $T$, denoted $T'$, is the left invertible given by
	\[
	T':= T (T^* T)^{-1} = {T^\dagger}^*.
	\]
\end{defn}

\begin{prop}
	\label{Cauchy dual}
	Let $T$ be a left invertible operator, and  $T' $ its Cauchy dual. The following statements hold:
	
	\begin{enumerate}[i.]
		\item $T'$ is left invertible with Moore-Penrose inverse ${T'}^\dagger = T^*$
		\item $\mathscr{E}':= \ker((T')^*) = \ker(T^\dagger) = \ker(T^*) = \mathscr{E}$
		\item $\mbox{\ind}(T') = \mbox{\ind}(T)$
	\end{enumerate}
\end{prop}

\begin{proof}
	It is clear from the definition that $T'$ is left invertible with $T^*$ a left inverse.  That ${T'}^\dagger = T^*$ follows from a simple computation:
	\[
	{T'}^\dagger = ({T'}^* T')^{-1} {T'}^* = (T^\dagger T')^{-1} T^\dagger = (T^* T) T^\dagger = T^*.
	\]
	The remaining observations now follow. 
\end{proof}
For the Cauchy dual $T'$, we define the analogous invariant subspaces:
\[
\begin{array}{rcl}
\HH_I' &:=& \bigcap_{n \geq 1} {T'}^n \HH\\
\HH_A' &:=&  \bigvee_{n \geq 0} {T'}^n \mathscr{E}.
\end{array}
\]

While one cannot hope to arrive at a decomposition $\HH = \HH_I \oplus \HH_A$, there is a duality between the spaces $\HH_I, \HH_I'$ and $\HH_A, \HH_A'$.  

\begin{prop}[\cite{Shimorin}, Prop 2.7]
	\label{anal decomp}
	Let $T$ be a left invertible operator.  Then 
	\[
	\HH = \HH_I \oplus \HH_A' = \HH_I' \oplus \HH_A.
	\]
	where $\oplus$ is an orthogonal direct summand of closed subspaces. 
\end{prop}
This duality is key in analyzing $\A_T$. We will leverage information between $T$ and $T'$ (or $T^\dagger$ and $T^*$) in order to prove theorems about $\A_T$.  

For some isometries, the Wold-decomposition is trivial.  For example, the unilateral shift on $\ell^2(\mathbb{N})$ is purely isometric since the subspace $\HH_I = 0$. This leads us to the following definition:
\begin{defn}[\cite{Shimorin}]
	An operator $T \in \BH$ is  \textbf{analytic} if $\HH_I = 0$.
\end{defn}
We will show in Section 3 that when $T$ is analytic, $T$ is unitarily equivalent to $M_z$ on a reproducing kernel Hilbert space of analytic functions.

%%%%%%%%%%%%%%%%%%%%%%%%%%%%%%%%%%%%%%%%
%%%%%%%%%%%%%%%%%%%%%%%%%%%%%%%%%%%%%%%%
%%%%%%%%%%%%%%%%%%%%%%%%%%%%%%%%%%%%%%%%
%%%%%%%%%%%%%%%%%%%%%%%%%%%%%%%%%%%%%%%%
\subsection{Basis and Dual Basis}
%%%%%%%%%%%%%%%%%%%%%%%%%%%%%%%%%%%%%%%%
%%%%%%%%%%%%%%%%%%%%%%%%%%%%%%%%%%%%%%%%
%%%%%%%%%%%%%%%%%%%%%%%%%%%%%%%%%%%%%%%%
%%%%%%%%%%%%%%%%%%%%%%%%%%%%%%%%%%%%%%%%

We now explore how $T \mid_{\HH_A}$ acts as a shift on a general basis. This will be done by showing if $T$ is a natural analytic left invertible, then it endows the Hilbert space with a type of basis analogous to that of a (Hamel) basis for a vector space, called a Schauder basis.

\begin{defn}
	A Banach space $X$ is said to have a \textbf{Schauder basis} if there exists a sequence $\{x_n\}$ of $X$ such that for every element $x \in X$, there is a unique sequence of scalars $\alpha_n$ such that 
	\[
	x = \sum_{n \geq 0} \alpha_n x_n
	\]
	where the above sum is converging in the norm topology of $X$. Alternatively, $\{x_n\}$ is a Schauder basis if and only if 
	
	\begin{enumerate}[i.]
		\item $\overline{\mbox{\Span}}\{x_n\} = X$
		\item $\sum a_n x_n = 0$ if and only if $a_n = 0$ for all $n$.
	\end{enumerate}

\end{defn}

Recall that a subspace $\mathscr{E}$ is said to be a \textit{wandering subspace} for an operator $T \in \BH$ if for each $n \in \mathbb{N}$, $\mathscr{E} \perp T^n \mathscr{E}$ \cite{Halmos}.  In the case of isometric operators, one further has $T^n \mathscr{E} \perp T^m \mathscr{E}$ for each $n, m \in \mathbb{N}$ with $n \neq m$.

Let $T$ be a natural analytic left invertible operator, and $L$ be a left inverse of $T$.  The next result shows that $\mathscr{E} = \ker(T^*)$ is a wandering subspace for $T$ and $L^*$. However, $T^n \mathscr{E}$ may not be orthogonal to $T^m \mathscr{E}$ for $n \neq m$.  The invariant subspace generated in this fashion is the whole Hilbert space. Thus, the orbit of $T$ and $L^*$ on $\ker(T^*)$ give rise to a Schauder basis:

\begin{thm}
	\label{sbasis}
	Let $T$ be a natural analytic left invertible operator with $\mbox{\ind}(T) = -n$ for some positive integer $n$. Let $\{x_{i,0}\}_{i=1}^{n}$ be an orthonormal basis for $\ker(T^*)$, and $L$ be a left inverse of $T$. Then
	
	\begin{enumerate}[i.]
		\item $x_{i,j}:= T^j x_{i,0}$, $i = 1, \dots n$, $j = 0, 1, \dots$ is a Schauder basis for $\HH$
		\item $x_{i,j}':=(L^*)^j x_{i,0}$, $i = 1, \dots n$, $j = 0, 1, \dots$  is a Schauder basis for $\HH$.
	\end{enumerate}

\end{thm}

\begin{proof}
	We will only prove the case when $\mbox{\ind}(T) = -1$.  The general case is no more complicated, but simply requires extra notation for bookkeeping. In this case, $\ker(T^*) = \mbox{\Span}\{x_0\}$ for some norm one element $x_0 \in \HH$. 
	
	The proof will proceed as follows.  First we will show that the wandering space for $T':={T^\dagger}^*$ produces a Schauder basis.  Then we show that the orbit of $x_0$ under powers of $T$ will produce a Schauder basis, which will allow us to conclude that for any left inverse $L$, the orbit of $L^*$ yields a Schauder basis.

	Since $T$ is analytic, by Proposition \ref{anal decomp}, we have that 
	\begin{equation}
	\label{hspan}
	\HH = \HH_A' = \bigvee_{j \geq 0} {T'}^j \ker(T^*).
	\end{equation}
	Let $x_j':= {T'}^j x_0$ for $j = 0, 1, \dots$.  Then by construction, $T' x_j' = x_{j+1}'$ and
	\[
	{T^*}^{m} x_{j}' = \left\{ \begin{array}{ll} 0 & \mbox{ if } m > j\\ x_{j-m}' & \mbox{ if } m \leq j \end{array} \right.
	\]
	
	Notice that \{$x_j'\}$ is a Schauder basis.  Indeed by (\ref{hspan}), $\overline{\mbox{\Span}}\{x_j'\} = \HH$.  Furthermore, if $\sum_{j \geq 0} a_j x_j' = 0$, then 
	\begin{equation}
	\label{uniqueco}
	0 = (I-T T^\dagger){T^*}^{m} \left(\sum_{j \geq 0} a_j x_j' \right) = (I-T T^\dagger) \left( \sum_{j \geq m} a_{j} x_{j-m}' \right) = a_{m} x_0.
	\end{equation}
	Thus, $a_j = 0$ for all $j$. Therefore $\{x_j'\}$ form a Schauder basis.
	
	We now show that $x_j:= T^j x_0$ is a Schauder basis.  Let $\KK$ be the closed subspace of $\HH$ given by $\KK:=\overline{\mbox{\Span}_{j \geq 1}}\{x_j\}$.  Suppose that $z \perp \KK$.  Then by above, $z$ has a unique expansion in the Schauder basis $x_j'$.  Say, $z = \sum_{j \geq 0} b_j x_j'$. Thus,
	\[
	0 = \langle z, x_m \rangle = \langle {T^*}^m z, x_0 \rangle = \langle {T^*}^m z, (I - T T^\dagger)x_0 \rangle = \langle (I - T T^\dagger) {T^*}^m z, x_0 \rangle = b_m.
	\] 
	Hence, $b_j = 0$ for all $j$, so $z = 0$.  Therefore, $\KK= \HH$.  Now suppose that $\sum_{j \geq 0} c_j x_j = 0$.  Then the exact same argument appearing in Equation (\ref{uniqueco}) with ${T^*}^m$ replaced with ${T^\dagger}^m$ shows $c_j = 0 $ for all $j$.
	
	Finally, suppose $L$ is any left inverse of $T$.  Let $y_j = {L^*}^j x_0$. Replacing the roles of $x_j$ with $y_j$ and $x_j'$ with $x_j$ in the preceding paragraph, one concludes that  $y_j$ is a Schauder basis for $\HH$.
\end{proof}

\begin{cor}
	\label{T T' anal}
	Let $T \in \BH$ be a natural left invertible.  Then $T$ is analytic if and only if $T'$ is analytic. 
\end{cor}

\begin{proof}
	If $T$ is analytic, then by Theorem \ref{sbasis}, $\HH_A = \HH$.  Hence $\HH_I' = 0$ by Proposition \ref{anal decomp}.  The converse statement is identical. 
\end{proof}

Theorem \ref{sbasis} illustrates how to construct Schauder bases for $\HH$ using a natural analytic left invertible operator $T$ and its Cauchy dual.  We reserve the notation of Theorem \ref{sbasis} for these bases. We make the following definition:
\begin{defn}
	Let $T$ be a natural analytic left invertible operator and $L$ be a left inverse of $T$.  Fix an orthonormal basis $\{x_{i,0}\}_{i=1}^n$ for $\mathscr{E} = \ker(T^*)$. Then
	\begin{equation}
	\label{Tbasis}
	\begin{array}{rcl}
	x_{i,j}&:=& T^j x_{i,0}\\
	x_{i,j}'&:=& {L^*}^j x_{i,0}.
	\end{array}
	\end{equation}
	We refer to the Schauder basis $\{x_{i,j}\}$ in Equation (\ref{Tbasis}) as the \textbf{basis of T with respect to $\{x_{i,0}\}_{i=1}^n$}.  Similarly, we refer to the basis $\{x_{i,j}'\}$ as the \textbf{dual basis of T with respect to $\{x_{i,0}\}_{i=1}^n$ and $L$}.

\end{defn}
If no mention is made to the choice of left inverse $L$, it is assumed that $L = T^\dagger$. While the above definition depends on the choice of orthonormal basis $\{x_{i,0}\}_{i=1}^n$ for $\mathscr{E}$, we will usually refer to each as the \textit{basis of $T$} and \textit{dual basis of $T$} without reference.

By definition of a Schauder basis, for each $f \in \HH$, there exists a unique sequences of scalars $\{\alpha_{i,j}\}$ and $\{\alpha_{i,j}'\}$ such that 
\[
f = \sum_{j \geq 0} \sum_{i=1}^n \alpha_{i,j} x_{i,j} =  \sum_{j \geq 0} \sum_{i=1}^n \alpha_{i,j}' x_{i,j}'.
\]
Naturally, one would like to have a relationship between $\{\alpha_{i,j}\}$ or $\{\alpha_{i,j}'\}$ in terms of the element $f \in \HH$.  We have the following useful characterization:

\begin{prop}
	\label{expansion}
	For each $f \in \HH$, we have the following expansions:
	\[
	f = \sum_{j \geq 0} \sum_{i=1}^n \langle f, x_{i,j}' \rangle x_{i,j} =  \sum_{j \geq 0} \sum_{i=1}^n \langle f, x_{i,j} \rangle x_{i,j}'.
	\]
\end{prop}

\begin{proof}
	Suppose that $f = \sum_{j \geq 0} \sum_{i=1}^n \alpha_{i,j} x_{i,j}$.  Now, ${T^\dagger}^m x_{i,j} = 0$ if $j \leq m$ and $x_{i, j-m}$ otherwise.  Also, 	since $\{x_{i,0}\}$ is an orthonormal basis for $\ker(T^*)$, we have for each $m \geq 0$,
	\[
	\langle f, x_{i,m}' \rangle = \langle {T^\dagger}^m f, x_{i,0} \rangle = \alpha_{i,m}.
	\]
  The same argument shows that if we expand $f$ in terms of the dual basis of $T$ as $f = \sum_{j \geq 0} \sum_{i=1}^n \alpha_{i,j}' x_{i,j}'$, then $\alpha_{i,m}' = \langle f, x_{i,m} \rangle$. 
\end{proof}

\begin{cor}
	\label{bi ortho}
	The basis of $T$ is bi-orthogonal to the dual basis of $T$.  That is, $\langle x_{l,m}, x_{i,j}' \rangle = \delta_{l,i} \delta_{m,j}$
\end{cor}

\begin{proof}
	By Proposition \ref{expansion}, we have that
	\[
	x_{l,m} = \sum_{j \geq 0} \sum_{i=1}^n \langle x_{l,m}, x_{i,j}' \rangle x_{i,j}.
	\]
	However by definition, Schauder bases have a unique expansion in terms of the basis.  Hence, $\langle x_{l,m}, x_{i,j}' \rangle = 0$ unless $i = l$ and $j = m$. 
\end{proof}

Briefly, we would like to caution the reader about the order of basis and dual basis of $T$. A convergent series $\sum_{n \geq 0} x_n$ in a Banach space $X$ is said to be \textit{unconditionally convergent} if for every permutation $\sigma$ of $\mathbb{N}$, the series $\sum_{n \geq 0} x_{\sigma(n)}$ converges.  Otherwise, the series is said to be \textit{conditionally convergent}.  A Schauder basis $\{x_n\}$ in a Banach space $X$ is said to be a \textit{unconditional basis} if the series expansion $x = \sum_{n \geq 0} \alpha_n x_n$ is unconditional for every $x \in X$.  Otherwise, the basis is said to be \textit{conditional}.  Examples of unconditional bases for Hilbert spaces include orthonormal bases, and more generally, some frames. 

Unfortunately, all infinite dimensional Banach spaces with a basis must have conditional bases \cite{Pelczynski}. What is worse, verifying that a basis is unconditional is, in general, a very difficult task. Explicit constructions of conditional bases exist for Hilbert spaces.  Indeed, there is a class of examples for $L^2(\mathbb{T})$ of the form $\{e^{2\pi i nt} \phi(t)\}_{n \in \mathbb{Z}}$ for some $\phi \in L^2(\mathbb{T})$ (See \cite{Singer}, Example 11.2). From the author's perspective, it is not clear when the basis and dual basis of $T$ are unconditional.  Fortunately, this will not affect our analysis in any serious way.   At a minimum, we have the following trivial rearrangements:
\begin{prop}
	\label{order}
	Let $T$ be a natural analytic left invertible with $\mbox{\ind}(T) = -n$ for some $1 \leq n < \infty$.  Then for any permutation $\sigma$ of $\{1, \dots, n\}$, we have
	\[
	\sum_{j \geq 0} \sum_{i=1}^n \alpha_{i,j} x_{i,j} = \sum_{i=1}^n \sum_{j \geq 0}  \alpha_{i,j} x_{i,j} = \sum_{i=1}^n \sum_{j \geq 0}  \alpha_{\sigma(i),j} x_{\sigma(i),j} =  \sum_{j \geq 0} \sum_{i=1}^n  \alpha_{\sigma(i),j} x_{\sigma(i),j}
	\]
	whenever the sum converges.  Consequently, $\sum_{j \geq 0} \sum_{i=1}^n \alpha_{i,j} x_{i,j}$ converges if and only if $\sum_{j \geq 0}  \alpha_{i,j} x_{i,j}$ converges for each $i = 1, \dots , n$. 
\end{prop}
%This remark is useful when we construct a canonical model for $T$ as multiplication by $z$ on a reproducing kernel Hilbert space of analytic functions in Section Three.  In order to conduct a more thorough analysis of $\A_T$, we will later consider the case when $\mbox{\ind}(T) = -1$. 

%%%%%%%%%%%%%%%%%%%%%%%%%%%%%%%%%%%%%%%%%%%%%%%%%%%%%%%%%%
%%%%%%%%%%%%%%%%%%%%%%%%%%%%%%%%%%%%%%%%%%%%%%%%%%%%%%%%%%
%%%%%%%%%%%%%%%%%%%%%%%%%%%%%%%%%%%%%%%%%%%%%%%%%%%%%%%%%%
%%%%%%%%%%%%%%%%%%%%%%%%%%%%%%%%%%%%%%%%%%%%%%%%%%%%%%%%%%
%%%%%%%%%%%%%%%%%%%%%%%%%%%%%%%%%%%%%%%%%%%%%%%%%%%%%%%%%%
%%%%%%%%%%%%%%%%%%%%%%%%%%%%%%%%%%%%%%%%%%%%%%%%%%%%%%%%%%
%%%%%%%%%%%%%%%%%%%%%%%%%%%%%%%%%%%%%%%%%%%%%%%%%%%%%%%%%%
%%%%%%%%%%%%%%%%%%%%%%%%%%%%%%%%%%%%%%%%%%%%%%%%%%%%%%%%%%
\section{Cowen-Douglas Operators - The Analytic Model}
%%%%%%%%%%%%%%%%%%%%%%%%%%%%%%%%%%%%%%%%%%%%%%%%%%%%%%%%%%
%%%%%%%%%%%%%%%%%%%%%%%%%%%%%%%%%%%%%%%%%%%%%%%%%%%%%%%%%%
%%%%%%%%%%%%%%%%%%%%%%%%%%%%%%%%%%%%%%%%%%%%%%%%%%%%%%%%%%
%%%%%%%%%%%%%%%%%%%%%%%%%%%%%%%%%%%%%%%%%%%%%%%%%%%%%%%%%%
%%%%%%%%%%%%%%%%%%%%%%%%%%%%%%%%%%%%%%%%%%%%%%%%%%%%%%%%%%
%%%%%%%%%%%%%%%%%%%%%%%%%%%%%%%%%%%%%%%%%%%%%%%%%%%%%%%%%%
%%%%%%%%%%%%%%%%%%%%%%%%%%%%%%%%%%%%%%%%%%%%%%%%%%%%%%%%%%
%%%%%%%%%%%%%%%%%%%%%%%%%%%%%%%%%%%%%%%%%%%%%%%%%%%%%%%%%%

In the late 70s, Cowen and Douglas discovered that operators possessing an open set of eigenvalues can be associated with a particular Hermitian holomorphic bundle  \cite{Cowen}, \cite{Cowen2}. These operators, now called Cowen-Douglas operators, could in some cases be completely classified by simple geometric properties. For example, when the rank of the bundle is one, the curvature serves as a complete set of unitary invariants \cite{Cowen2}.

Cowen-Douglas operators have played an important role in operator theory, servicing as a bridge between operator theory  and complex geometry.  The definition is rigid enough to allow for classification based on local spectral data.  However, the definition is also flexible enough to allow for rich examples - including many backward weighted shifts and adjoints of some subnormal operators. The definition of Cowen-Douglas operators is as follows: 

\begin{defn}
	\label{CDdef}
	Given an open subset $\Omega$ of $\C$ and a positive integer $n$, we say that $R$ is of \textbf{Cowen-Douglas class n}, and write $R \in B_n(\Omega)$ if
	
	\begin{enumerate}[i.]
		\item $\Omega \subset \sigma(R)$
		\item $(R-\lambda)\HH = \HH$ for all $\lambda \in \Omega$
		\item $\dim(\ker(R-\lambda)) = n$ for all $\lambda \in \Omega$
		\item $\bigvee_{\lambda \in \Omega} \ker(R-\lambda) = \HH$
	\end{enumerate}
	
\end{defn}

Thus if $R \in B_n(\Omega)$, then $R$ contains an open set of eigenvalues such that each eigenspace has dimension $n$, and the span of these eigenspaces is dense in $\HH$. Associated to Cowen-Douglas operators is a bundle structure known as a Hermitian holomorphic vector bundle.  

\begin{defn}
A \textbf{Hermitian holomorphic vector bundle of rank $n$ over $\Omega$} consists of the following data:
\begin{enumerate}[i.]
	\item A complex manifold $E$
	\item A holomorphic map $\pi: E \rightarrow \Omega$ such that each fiber $E_\lambda:=\pi^{-1}(\lambda)$ is isomorphic to $\C^n$
	\item For each $\lambda_0 \in \Omega$, there exists a neighborhood $\Delta$ of $\lambda_0$ and functions $\{\gamma_i\}_{i=1}^n$ with $\gamma_i:\Omega \rightarrow E$ such that $\{\gamma_i(\lambda)\}_{i=1}^n$ form a basis for $E_\lambda$.
\end{enumerate}
A \textbf{cross-section} $E$ is a map $\gamma: \Omega \rightarrow E$ such that $\pi(\gamma(\lambda)) = \lambda$ for all $\lambda \in \Omega$.  The bundle is \textbf{trivial} if $\Delta$ may be taken to be $\Omega$.  The \textbf{trivial bundle of rank $n$ over $\Omega$} is $\Omega \times \C^n$ with $\pi(\lambda, x) = \lambda$. 
\end{defn}

If $R \in B_n(\Omega)$, then the set
\[
E_R:= \{ (\lambda, x) \in \Omega \times \HH: x \in \ker(R - \lambda) \}
\]
with the mapping $\pi: E_R \rightarrow \Omega$ via $\pi(\lambda, x) = \lambda$ defines sub-bundle of the trivial bundle of rank $n$ over $\Omega$. It is known that $E_R$ provides a complete set of unitary invariants for operators in the Cowen-Douglas class \cite{Cowen}.  Specifically, if $E_{R_1}$ is isomorphic to $E_{R_2}$ as holomorphic vector bundles, then $R_1$ is unitarily equivalent to $R_2$.  This approach to Cowen-Douglas theory highlights the beautiful connections that exist between complex geometry and operator theory.  

The sections of the bundle $E_R$ provide an equivalent avenue of study.  Given $R \in B_n(\Omega)$, we can represent $R$ as the adjoint of  multiplication by $z$ on a reproducing kernel Hilbert space.  The approach of this paper more closely follows this model.  We will outline this construction below, and connect it to our work on bases in Section Two. For more information about Cowen-Douglas operators, see \cite{Cowen}, \cite{Curto}, \cite{Zhu}.

%%%%%%%%%%%%%%%%%%%%%%%%%%%%%%%%%%%%%%%%%%%%%%%%%%%%%%%%%%
%%%%%%%%%%%%%%%%%%%%%%%%%%%%%%%%%%%%%%%%%%%%%%%%%%%%%%%%%%
%%%%%%%%%%%%%%%%%%%%%%%%%%%%%%%%%%%%%%%%%%%%%%%%%%%%%%%%%%
%%%%%%%%%%%%%%%%%%%%%%%%%%%%%%%%%%%%%%%%%%%%%%%%%%%%%%%%%%
\subsection{Analytic Left Invertibles and Cowen-Douglas Operators}
%%%%%%%%%%%%%%%%%%%%%%%%%%%%%%%%%%%%%%%%%%%%%%%%%%%%%%%%%%
%%%%%%%%%%%%%%%%%%%%%%%%%%%%%%%%%%%%%%%%%%%%%%%%%%%%%%%%%%
%%%%%%%%%%%%%%%%%%%%%%%%%%%%%%%%%%%%%%%%%%%%%%%%%%%%%%%%%%
%%%%%%%%%%%%%%%%%%%%%%%%%%%%%%%%%%%%%%%%%%%%%%%%%%%%%%%%%%

The connection between Cowen-Douglas operators and left invertibles is found in the following:

\begin{thmy}
	\label{CD}
	Let $T \in \BH$ be a left invertible operator with $\mbox{\ind}(T) = -n$, for $n \geq 1$.  Then the following are equivalent:
	
	\begin{enumerate}[i.]
		\item $T$ is analytic
		\item $T'$ is analytic
		\item There exists $\epsilon > 0$ such that $T^* \in B_n(\Omega)$ for $\Omega = \{z: |z|<\epsilon\}$
		\item There exists $\epsilon > 0$ such that $T^\dagger \in B_n(\Omega)$ for $\Omega = \{z: |z|<\epsilon\}$
	\end{enumerate}
	
\end{thmy}

Theorem \ref{CD} is a cornerstone result for this work.  It serves two fundamental roles.  First, Theorem \ref{CD} allows us to leverage the powerful machinery associated with Cowen-Douglas operators into classifying the algebras $\A_T$.  Second, it provides us with a desirable canonical model. Concretely, Theorem \ref{CD} allows us to represent $T$ as multiplication by $z$ restricted to a reproducing kernel Hilbert space of analytic functions. 

To help illuminate this relationship, we will take a constructive approach to proving Theorem \ref{CD}.  This will also connect to our results on Schauder bases from the previous section.  We prove the implication \textit{iii.} implies \textit{i.} after stating the following lemma noted in Cowen and Douglas' original work:

\begin{lemma}[\cite{Cowen}]
	\label{Li lemma}
	Let $\Theta$ be an open subset of $\C$ and $S \in B_m(\Theta)$.  Then for any fixed $\mu_0 \in \Theta$, 
	\[
	\bigvee_{k \geq 1} \ker(S - \mu_0)^k = \HH.
	\]
	Moreover, if $\Omega \subset \mathbb{C}$ is open, $\lambda_0 \in \Omega$, $n$ is a positive integer, and $R \in \BH$ satisfies 
	
	\begin{enumerate}[i.]
		\item $\Omega \subset \sigma(S)$
		\item  $(R-\lambda)\HH = \HH$ for all $\lambda \in \Omega$
		\item $\dim(\ker(R-\lambda)) = n$ for all $\lambda \in \Omega$
		\item $\bigvee_{k \geq 1} \ker(R - \lambda_0)^k = \HH$.
	\end{enumerate}
	
	\noindent	Then $R \in B_n(\Omega)$. 
\end{lemma}

\begin{cor}
	\label{CD implies anal}
	Let $T \in \BH$, $n \in \mathbb{N}$, $\epsilon > 0$ and $\Omega = \{z: |z|<\epsilon\}$.  If $T^*\in B_n(\Omega)$, then $T$ is an analytic, left invertible operator with $\mbox{\ind}(T) = -n$.
\end{cor}

\begin{proof}
	By assumption, $0 \in \Omega \subset \sigma(T^*)$.  By condition \textit{ii.} of the definition of Cowen-Douglas operators, $T^*$ is onto.  Therefore, $T$ is left invertible. 
	
	As $T$ is left invertible, its Cauchy dual $T'$ is well defined.  Recall that $T^* = (T')^\dagger$.  Since $T^* \in B_n(\Omega)$, it follows that $\mbox{\ind}(T') = -n$. By Proposition \ref{Cauchy dual} and condition \textit{iii.} of Cowen-Douglas operators, we have $\mbox{\ind}(T) = \mbox{\ind}(T') = -n$.   Thus, all that remains to be shown is that $T$ is analytic.  By lemma \ref{Li lemma}, $\HH = \bigvee_{k \geq 1} \ker ({T^*}^k)$.  Therefore,
	\[
	0 = \left(\bigvee_{k \geq 1} \ker ({T^*}^k) \right)^\perp = \bigcap_{k \geq 1} \ker ({T^*}^k)^\perp = \bigcap_{k \geq 1} \mbox{ran}(T^k). \qedhere
	\]
\end{proof}

Next we show that if $T$ is a natural analytic left invertible, then $T^* \in B_n(\Omega)$.  This will be done in several steps. First, we will show that $T^*$ possess an open set $\Omega$ of eigenvalues.  We establish some notation for the open set $\Omega$ that will appear in the implication \textit{i.} implies \textit{iii.} of Theorem \ref{CD}:

\begin{defn}
	\label{omega}
	Suppose $T$ is a natural analytic left invertible operator.   We define
	\[
	\Omega_T := \{z \in \mathbb{C}: |z|< \|T^\dagger\|^{-1}\}.
	\]
\end{defn}
It follows from Proposition \ref{left_facts} and Proposition \ref{close ind} that if $T$ is a natural analytic left invertible operator, and $\lambda \in \Omega_T$, then $T+\lambda$ is left invertible with $\ind(T) = \ind(T+\lambda)$.  Furthermore, if $\lambda \in \Omega_T$ then since  $|\lambda| < \|T^\dagger\|^{-1}$ and $T' = {T^\dagger}^*$, the operator $\lambda T'$ has norm less than 1.  Therefore, the operator  $I - \lambda T'$ is invertible with
	\[
	(I - \lambda T')^{-1} = \sum_{j \geq 0} \lambda^j {T'}^j.
	\]

\begin{lemma}
	\label{choice}
	Let $T$ be an analytic left invertible operator with $\mbox{\ind}(T) = -n$ for some positive integer $n$.   Let $\{x_{i,0}\}_{i=1}^n$ be an orthonormal basis for $\ker(T^*)$, and 
	\[
	x_{i,j}' = {T'}^j x_{i,0} = ((T^\dagger)^*)^j x_{i,0}
	\]
	be the dual basis of $T$ with respect to $T^\dagger$. Then for each $i=1, \dots, n$, the maps $\gamma_i:\Omega_T \rightarrow \HH$ via
	\[
	\gamma_i(\lambda):= \sum_{j \geq 0} \lambda^j x_{i,j}'
	\]
	are well defined.  Furthermore, the maps $\gamma_i:\Omega_T \rightarrow \HH$ are analytic.
\end{lemma}

\begin{proof}
	The operator $I - \lambda T'$ is invertible by construction.  Thus for each $i = 1, \dots n$,
	\[
	(I-\lambda T')^{-1} \left( x_{i,0} \right) = \sum_{j \geq 0} \lambda^j {T'}^j \left( x_{i,0} \right)  =  \sum_{j \geq 0} \lambda^j x_{i,j}' = \gamma_i(\lambda)
	\]
	exists for each $\lambda \in \Omega_T$.  Since the map $\lambda \mapsto (I - \lambda T')^{-1}$ is well defined and analytic on $\Omega_T$, we have that the maps $\gamma_i$ are analytic.
\end{proof}

In light of this observation, we reserve the following notation:

\begin{defn}
	\label{section}
	Given an analytic left invertible $T$ with $\mbox{\ind}(T) = -n$ for some positive integer $n$, let $\Omega_T$ be as in  Definition \ref{omega}.    Let $\{x_{i,0}\}_{i=1}^n$ be an orthonormal basis for $\ker(T^*)$,  and $x_{i,j}' = {T'}^j x_{i,0}$ be the dual basis of $T$ with respect to $T^\dagger$.  We define
	\[
	\gamma_i(\lambda):= \sum_{j \geq 0} \lambda^j x_{i,j}'.
	\]
	
\end{defn}

\begin{lemma}
	\label{evectors}
	Let $T$ be an analytic left invertible with $\mbox{\ind}(T) = -n$, and $\{\gamma_i\}_{i=1}^n$ be as in Definition \ref{section}.  Then for each $i$,
\[
\gamma_i(\lambda) \in \ker(T^* - \lambda).
\]
Hence, $\Omega_T \subset \sigma_p(T^*)$.
\end{lemma}

\begin{proof}
	Since $T^*$ is the Moore-Penrose inverse of $T'$, it follows from the definition of $\gamma_i$ that
	\[
	T^* \gamma_i(\lambda) = \sum_{j=0}^\infty \lambda^j T^* x_{i,j}' = \sum_{j=1}^\infty \lambda^j x_{i,j-1}' = \lambda \gamma_i(\lambda).
	\]
	The rest of the statement follows. 
\end{proof}

\begin{prop}
	\label{equiv}
	Let $T$ be an analytic left invertible operator  with $\mbox{\ind}(T) = -n$ for some positive integer $n$.  Let $\Omega_T$ be be as in  Definition \ref{omega}. Then $T^* \in B_n(\Omega_T)$
\end{prop}

\begin{proof}
	Pick an orthonormal basis $\{x_{i,0}\}$ for $\ker(T^*)$.   If $\lambda \in \Omega_T$, then $T-\lambda$ is left invertible with Fredholm index $-n$.   Therefore, each eigenspace $\ker(T^*-\lambda)$ is n-dimensional for each $\lambda \in \Omega_T$. By Lemma \ref{evectors}, we have $\{\gamma_i(\lambda)\}_{i=1}^n \subset \ker(T^*-\lambda)$. Moreover, since $\{x_{i,j}'\}$ form a Schauder basis, we must have that the collection $\{\gamma_i(\lambda)\}_{i=1}^n$ is linearly independent.    
	
	Indeed, suppose there exists a $\mu \in \C$ such that $\gamma_i(\lambda) = \mu \gamma_k(\lambda)$ for some $\lambda \in \Omega_T$ with $i \neq k$.  If $x_{i,j} = T^j x_{i,0}$ is the basis associated to $T$, then by Lemma \ref{bi ortho} we have for each $j$ 
	\[
	\lambda^j = \langle \gamma_i(\lambda), x_{i,j} \rangle = \mu \langle \gamma_k(\lambda), x_{i,j} \rangle = \mu \sum_{j=0}^\infty \lambda^j \langle x_{k,j}',x_{i,j} \rangle =   0.
	\]
	This forces $\lambda = 0$. Hence, $x_{i,0}' = \gamma_i(0) = \mu \gamma_k(0) = \mu x_{k,0}'$.  But since $\{x_{i,0}\}_{i=1}^n$ form an orthonormal basis for $\ker(T^*)$, this cannot happen. Hence, $\{\gamma_i(\lambda)\}_{i=1}^n$ form a (perhaps non-orthogonal) basis for $\ker(T^*-\lambda)$.
	
	Lastly, if we choose $\lambda_0 = 0$, then 
	\[
	\ker(T^* - \lambda_0)^k = \ker((T^*)^k) = (\mbox{ran}{T^k})^\perp =  \left( \bigcap_{j=0}^k T^j \HH \right)^\perp.
	\]
	Since $T$ is analytic, it follows that $\bigvee_{k \geq 1} \ker((T^*)^k) =  \HH$. By Lemma \ref{Li lemma}, we have that $T^* \in B_n(\Omega)$.
\end{proof}

We highlight an important and interesting feature of the basis $\{x_{i,j}'\}$ that came up in the previous proof:

\begin{cor}
	\label{span ker}
 	Let $T$ be an analytic left invertible operator  with $\mbox{\ind}(T) = -n$ for some positive integer $n$, and $\{\gamma_i\}_{i=1}^n$ be the analytic maps from Definition \ref{section}.  Then for each $\lambda \in \Omega_T$, $\{\gamma_i(\lambda)\}_{i=1}^n$ form a spanning set for $\ker(T^*-\lambda)$.
\end{cor}

We have thus shown that statements \textit{i} and \textit{iii.} of Theorem \ref{CD} are equivalent.    However, when paired with Corollary \ref{T T' anal} we see that $T^\dagger$ must also be Cowen-Douglas.  This completes the proof of Theorem \ref{CD}.

One consequence of Theorem \ref{CD} is a reformulation of the definition of $\A_T$ and the operator algebra generated by a Cowen-Douglas operator and a particular right inverse.  Indeed, recall that $\A_T$ is defined by
\[
\A_T:= \overline{\mbox{\Alg}}\{T, T^\dagger \}.
\]
If $\epsilon >0$, $\Omega = \{z: |z|<\epsilon\}$ , and $R \in B_n(\Omega)$, then by definition $R$ is right invertible.  There exists a canonical right inverse of $R$, which we denote by $T$, such that $\mbox{ran}(T) = \ker(R)^\perp$.  By construction, $T$ is left invertible, and $R = T^\dagger$, the Moore-Penrose inverse of $T$.    Thus, we arrive at an equivalent viewpoint of study:
\begin{cor}
	Let $\epsilon > 0$, $\Omega = \{z: |z|<\epsilon\}$, and $R \in B_n(\Omega)$. If $T$ is the right inverse of $R$ such that $\mbox{ran}(T) = \ker(R)^\perp$, then $T$ is an analytic left invertible operator with $R = T^\dagger$.  Hence, 
	\[
	\A_T = \overline{\mbox{\Alg}}\{T, R\}.
	\]
\end{cor}

%%%%%%%%%%%%%%%%%%%%%%%%%%%%%%%%%%%%%%%%%%%%%%%%%%%%%%%%%%
%%%%%%%%%%%%%%%%%%%%%%%%%%%%%%%%%%%%%%%%%%%%%%%%%%%%%%%%%%
%%%%%%%%%%%%%%%%%%%%%%%%%%%%%%%%%%%%%%%%%%%%%%%%%%%%%%%%%%
%%%%%%%%%%%%%%%%%%%%%%%%%%%%%%%%%%%%%%%%%%%%%%%%%%%%%%%%%%
\subsection{The Associated Reproducing Kernel Hilbert Space}
%%%%%%%%%%%%%%%%%%%%%%%%%%%%%%%%%%%%%%%%%%%%%%%%%%%%%%%%%%
%%%%%%%%%%%%%%%%%%%%%%%%%%%%%%%%%%%%%%%%%%%%%%%%%%%%%%%%%%
%%%%%%%%%%%%%%%%%%%%%%%%%%%%%%%%%%%%%%%%%%%%%%%%%%%%%%%%%%
%%%%%%%%%%%%%%%%%%%%%%%%%%%%%%%%%%%%%%%%%%%%%%%%%%%%%%%%%%

As previously remarked, the general theory of Cowen-Douglas operators allows one to represent $T$ as multiplication by $z$ on a reproducing kernel Hilbert space of analytic functions over $\Omega$. This construction, when connected to the bases associated to $T$ and $T'$ in Section Two, provides an updated heurstic for the algebra $\A_T$.  First, let us establish some notation. Given a set $G \subset \C$, let $G^* := \{\overline{\lambda}: \lambda \in G\}$.  Notice that $\Omega_T^* = \Omega_T$ as a set.   We make the following definition:

\begin{defn}[\cite{Zhu}]
	Let $R \in B_n(\Omega)$.  A holomorphic cross-section of $\gamma: \Omega \rightarrow E_R$ of the bundle $E_R$ is a \textbf{spanning holomorphic cross-section} if 
	\[
	\overline{\mbox{span}}\{\gamma(\lambda):\lambda \in \Omega\} = \HH.
	\]
\end{defn}

Spanning holomorphic cross-sections give rise to reproducing kernel Hilbert spaces of analytic functions. Indeed, fix a spanning holomorphic section $\gamma$.  For each $f \in \HH$, define an analytic function   $\hat{f}_\gamma \in H({\Omega}^*)$ as follows:

\begin{equation}
\label{anal formula}
\hat{f}_\gamma(\lambda) = \langle f, \gamma(\overline{\lambda}) \rangle \quad \quad \quad \quad \quad \lambda \in \Omega^*.
\end{equation}
Let $\widehat{\HH}_\gamma = \{\hat{f}_\gamma: f \in \HH\} \subset H(\Omega^*)$.  Equip $\widehat{\HH}_\gamma$ with the inner product afforded by $\HH$.  That is, for each $f, g \in \HH$, define the inner product on $\widehat{\HH}_\gamma$ via
\[
\langle \hat{f}_\gamma, \hat{g}_\gamma \rangle_\gamma := \langle f, g \rangle.
\]
Define a linear map $U_\gamma: \HH \rightarrow \widehat{\HH}_\gamma$ via $U_\gamma(f) = \hat{f}_\gamma$.  Notice that because $\gamma$ is a spanning section, $U_\gamma$ is a unitary.  Indeed, if $\hat{f}_\gamma = \hat{g}_\gamma$, then for each $\lambda \in \Omega^*$,
\[
0 = \hat{f}_\gamma(\lambda) - \hat{g}_\gamma(\lambda) = \langle f - g, \gamma(\overline{\lambda})\rangle
\]
Since the span of 	$\{\gamma(\lambda):\lambda \in \Omega\}$ is dense in $\HH$, $f - g= 0$.

Furthermore, $\widehat{\HH}_\gamma$ is a reproducing kernel Hilbert space over the set $\Omega^*$.  Indeed, as $\gamma(\overline{\lambda}) \in \HH$, there exists a function $\widehat{\gamma(\overline{\lambda})}_\gamma \in \widehat{\HH}_\gamma$.  For all $f \in \HH$ and $\lambda \in \Omega$, 
\[
\hat{f}_\gamma(\lambda) = \langle f, \gamma(\overline{\lambda}) \rangle = \left\langle \hat{f}_\gamma, \widehat{\gamma(\overline{\lambda})}_\gamma \right\rangle_\gamma.
\]
Hence, the reproducing kernel at $\lambda \in \Omega^*$ is given by $k_{\lambda} = \widehat{\gamma(\overline{\lambda})}_\gamma$.  Therefore, given $\lambda, \mu \in \Omega^*$, the reproducing kernel may be computed as follows:
\[
K(\lambda, \mu) = \langle k_\mu, k_\lambda \rangle =  \left\langle \widehat{\gamma(\overline{\mu})}_\gamma,  \widehat{\gamma(\overline{\lambda})}_\gamma \right\rangle_\gamma = \langle\gamma(\overline{\mu}),  \gamma(\overline{\lambda}) \rangle.
\]

If $R \in B_n(\Omega)$, then the Hermitian holomorphic vector bundle $(E_R, \pi)$ has many choices of cross sections $\gamma: \Omega \rightarrow E_R$.  For example, if $T$ is a natural analytic left invertible,  the $\gamma_i$ in Definition \ref{section} are cross sections for $T^*$. By construction,  the collection of cross-sections $\{\gamma_i\}_{i=1}^n$ satisfy $\{\gamma_i(\lambda)\}_{i=1}^n$ form a basis for $E_\lambda$. Since the fibers $E_\lambda$ of $E_R$ are  $\ker(R-\lambda)$, and $\bigvee \ker(R-\lambda) = \HH$, we have that the collection of $\gamma_i:\Omega \rightarrow \HH$ have dense span in $\HH$.  The following theorem states that we can combine these sections to get a spanning holomorphic cross-section:

\begin{thm}[\cite{Zhu} - Theorem 5]
	\label{zhu theorem}
	Let $\HH$ be a Hilbert space, and $\{\gamma_i\}_{i=1}^n$ be holomorphic functions from $\Omega$ to $\HH$ such that
	\[
	\bigvee_{\lambda \in \Omega} \mbox{\Span}_{i=1, \dots, n}\{\gamma_i(\lambda)\} = \HH.
	\]
	  Then there exists holomorphic functions $\{\phi_i\}_{i=1}^n$ from $\Omega \rightarrow \C$ such that the map $\gamma: \Omega \rightarrow \HH$ defined by 
	\[
	\gamma(\lambda):= \sum_{i=1}^n \phi_i(\lambda) \gamma_i(\lambda) \quad \quad \quad \lambda \in \Omega
	\]
also spans $\HH$.
\end{thm}
The functions $\phi_i$ that appear in Theorem \ref{zhu theorem} are built as follows.  Let $\HH_1 = \bigvee_{\lambda \in \Omega} \gamma_1(\lambda)$.  Then by construction, $\gamma_1$ is a holomorphic spanning cross-section for $\HH_1$. Consider the RKHS of analytic functions built from $\gamma_1$.  One can find a set of points $\{a_l\} \subset \Omega$ that is a \textit{uniqueness set of $\Omega$}, in the sense that the only function in this space associated to $\gamma_1$ that vanishes on $\{a_l\}$ is the zero function. Using a separation theorem due to Weierstrass,  one can pick a holomorphic function $\phi_2$ that vanishes exactly on $\{a_l\}$. Then $\gamma_1 + \phi_2 \gamma_2$ ends up being a spanning section for the space $\HH_2 = \bigvee_{\lambda \in \Omega} \Span_{i=1,2}\{\gamma_i(\lambda)\}$.  Iteratively, one selects holomorphic functions $\phi_i$ until a spanning section for the whole Hilbert space is built. In particular, one can choose $\phi_1$ to be the identity function on $\Omega$.  For details, see \cite{Zhu}.

Notice that this construction is far from unique.  Indeed, $\gamma$ depends on a choice of uniqueness sets and functions $\{\phi_i\}$. Nevertheless, Theorem \ref{zhu theorem} provides a method for constructing spanning sections for all $R \in B_n(\Omega)$.

\begin{cor}
	\label{spanning}
	If $R \in B_n(\Omega)$, then $(E_R, \pi)$ admits a spanning holomorphic cross-section.
\end{cor}

Suppose $R \in B_n(\Omega)$.   A  consequence of Corollary \ref{spanning} is that $R$ is unitarily equivalent to multiplication by $z$ on a collection of analytic functions over $\Omega^*$.

Let $M_z$ denote the operator of multiplication by the indeterminate $z$.  That is, for each $\lambda \in \Omega^*$, $M_z(\hat{f}_\gamma)(\lambda) = \lambda \hat{f}_\gamma(\lambda)$.  Since $\overline{\lambda} \in \Omega$, it follows from the definition Cowen-Douglas operators that $\overline{\lambda}$ is an eigenvalue for $R$. Consequently,  $U_\gamma$ intertwines $M_z$ on $\widehat{\HH}_\gamma$ and $R^*$ on $\HH$.  Indeed for all $f \in \HH$,

\begin{equation}
\label{multiplication}
\begin{array}{rcl}
(U_\gamma R^* f)(\lambda) = \widehat{(R^*f)}_\gamma(\lambda) &=& \langle R^*f, \gamma(\overline{\lambda}) \rangle \\
&=& \langle f, R \gamma(\overline{\lambda}) \rangle\\
& =& \langle f, \overline{\lambda} \gamma(\overline{\lambda}) \rangle \\
&=& (M_z U_\gamma f)(\lambda).
\end{array}
\end{equation}
Thus, we have $U_\gamma R^* = M_z U_\gamma$, so $R^*$ is unitarily equivalent to $M_z$ on $\widehat{\HH}_\gamma$. 

In our current study of natural analytic left invertible operators, Theorem \ref{CD} says that $T^* \in B_n(\Omega_T)$.  Therefore, Equation (\ref{multiplication}) tells us that $T$ is unitarily equivalent to $M_z$ on $\widehat{\HH}_\gamma$. Furthermore, $\Omega_T = \Omega_T^*$ as sets, so for ease of notation, we consider the functions in $\widehat{\HH}_\gamma$ on $\Omega_T$.  We record this as a corollary.

\begin{cor}
	Let $T$ be an analytic, left invertible operator with $\mbox{\ind}(T) = -n$ for some positive integer $n$.  Then $T$ is unitarily equivalent to multiplication by $z$ on a reproducing kernel Hilbert space of analytic functions on $\Omega_T^* = \Omega_T$.  
\end{cor}

A natural question one might ask is, ``What are the analytic functions in $\widehat{\HH}_\gamma$ ''?  The answer will depend on the choice of analytic section $\gamma$ described above.  We will describe a salient representation $U_\gamma$ that blends together the Cowen-Douglas theory with the basis theory developed in Section Two.  

Let $\{x_{i,0}\}_{i=1}^n$ be an orthonormal basis for $\ker(T^*)$, and $\{\gamma_i\}_{i=1}^n$ be defined as in Definition \ref{section}.  By Corollary \ref{span ker} and Theorem \ref{zhu theorem}, there exists holomorphic functions $\{\phi_i\}_{i=1}^n$ from $\Omega \rightarrow \C$ such that 
\[
\gamma(\lambda):= \sum_{i=1}^n \phi_i(\lambda) \gamma_i(\lambda) = \sum_{i=1}^n \phi_i(\lambda) \sum_{j \geq 0} \lambda^j x_{i,j}'
\]
is a holomorphic spanning cross-section for $\HH$. By the comments following Theorem \ref{zhu theorem}, $\phi_1$ may be chosen to be the identity function.  For each $f \in \HH$ and $\lambda \in \Omega_T$, we have by Equation (\ref{anal formula})
\[
\hat{f}(\lambda) = \langle f , \gamma(\overline{\lambda}) \rangle =  \sum_{i=1}^n \phi_i(\lambda) \sum_{j \geq 0} \lambda^j \langle f, x_{i,j}'\rangle
\]
where here we have suppressed the subscript $\gamma$ on $\hat{f}$.  The reproducing kernel Hilbert space associated with this choice of analytic section will be simply denoted $\widehat{\HH}$.  We store this information in a definition:

\begin{defn}
	\label{rk_model}
	Given a natural analytic left invertible $T$, let $\Omega_T$ be as in  Definition \ref{omega}.    Let $\{x_{i,0}\}_{i=1}^n$ be an orthonormal basis for $\ker(T^*)$.  Pick $\{\phi_i\}_{i=1}^n$ holomorphic functions such that the map
\[
\gamma(\lambda)=\sum_{i=1}^n \phi_i(\lambda) \sum_{j \geq 0} \lambda^j x_{i,j}'
\]
each $\lambda \in \Omega_T$ is a spanning holomorphic cross-section with $\phi_1 = 1$.  For each  each $f \in \HH$, set
	\begin{equation}
	\label{usual}
	\hat{f}(\lambda) = \sum_{i=1}^n \phi_i(\lambda) \sum_{j \geq 0} \lambda^j \langle f, x_{i,j}'\rangle.
	\end{equation}
	Let $\widehat{\HH}$ denote the reproducing kernel Hilbert space of functions $\hat{f}$ arising from Equation (\ref{usual}) with inner product $\langle \hat{f}, \hat{g} \rangle = \langle f, g \rangle$. The representation of $T$ as $M_z$ on $\widehat{\HH}$ is called the \textbf{canonical representation of $T$ relative to $\{x_{i,0}\}_{i=1}^n$ and $\{\phi_i\}_{i=1}^n$}.
\end{defn}

The terminology canonical is fitting for the above representation. In the canonical representation,  the basis elements associated to $T$ become the functions $\phi_k z^l$.  That is, if $k = 1, \dots n$, then $\widehat{x_{k,l}}(\lambda) = \phi_k(\lambda) \lambda^l$ for each $\lambda \in \Omega$.   This follows directly by Corollary \ref{bi ortho} and Equation \eqref{usual}:
\begin{equation}
	\label{function basis}
\widehat{x_{k,l}}(\lambda) =  \sum_{i=1}^n \phi_i(\lambda) \sum_{j \geq 0} \lambda^j \langle x_{k,l}, x_{i,j}'\rangle =\phi_k(\lambda) \lambda^l
\end{equation}
In particular, since $\phi_1 = 1$, we have that $\widehat{\HH}$ contains the functions of the form $z^l$.  Furthermore, $\widehat{x_{k,0}} = \phi_k \in \widehat{\HH}$ for each $k=1, \dots, n$.  Since $\{x_{k,0}\}_{k=1}^n$ form an orthonormal basis for $\ker(T^*)$, the functions $\{\phi_k\}_{k=1}^n$ are also orthogonal.

Recall that in general, the reproducing kernel at $\lambda$ is given by $k_\lambda = \gamma(\overline{\lambda})$. Hence, for the canonical representation, the reproducing kernel $K:\Omega^2 \rightarrow \C$ for $\widehat{\HH}$ takes on the following form:
\[
K(\lambda, \mu) = 
\langle \gamma(\overline{\mu}), \gamma(\overline{\lambda}) \rangle = 
 \sum_{k=1}^n \sum_{i=1}^n \phi_i(\lambda) \overline{\phi_k(\mu)}\sum_{l \geq 0} \sum_{j \geq 0} \overline{\mu}^l \lambda^j  \langle x_{k,l}', x_{i,j}' \rangle
\]
where by Proposition \ref{order}, convergence does not depend on the order of the four sums. The kernel is analytic in $\lambda$, and co-analytic in $\mu$ by construction.

Under the canonical  representation, $T^\dagger$ becomes ``division by $z$''.  To make this precise, we require a simple lemma:

\begin{lemma}
	Let $T_1$ and $T_2$ be left invertible operators with Moore-Penrose inverses $T_1^\dagger$ and ${T_2}^\dagger$.  If $T_2 = U T_1 U^*$ for some unitary $U$, then ${T_2}^\dagger = U T_1^\dagger U^* = (U T_1 U^*)^\dagger$.
\end{lemma}

\begin{proof}
	Recall that ${T_2}^\dagger = ({T_2}^* T_2)^{-1} {T_2}^*$.  Hence,
	\[
	{T_2}^\dagger = (U T_1^* T_1 U^*)^{-1} U T_1^* U^* = U (T_1^* T_1)^{-1} U^* U T_1^* U^* = U T_1^\dagger U^*.
	\]
\end{proof}

\begin{cor}
	If $T$ is analytic with index $-n$, and $U_\gamma:\HH \rightarrow \widehat{\HH}_\gamma$ is the unitary such that $M_z = U_\gamma T U_\gamma^*$, then $M_z^\dagger = (U_\gamma T U_\gamma^*)^\dagger$.
\end{cor}

Now, the functions inside $\ker(M_z^\dagger)$ are the span of the orthogonal functions $\{\phi_i\}_{i=1}^n$. Furthermore, $\mbox{ran}(M_z) = \ker(M_z^\dagger)^\perp$ consists of functions of the form $z \hat{g}$.  From the preceding corollary, $M_z^\dagger M_z = I$, so it follows that either $M_z^\dagger \hat{f} = 0$ (if $\hat{f}$ is linear combination of the $\phi_i$) or $M_z^\dagger \hat{f} = z^{-1} \hat{f}$ otherwise. Expanding on this computation, suppose that $\hat{f} \in \widehat{\HH}$ is of the form $\phi_i z^j$. Consider the action of ${M_z^\dagger}^n$ on $\hat{f}$.   By construction, ${M_z^\dagger}^n(\phi_i z^j)(\lambda)$ is equal to $0$ if $n \geq j$ and $\phi_i z^{j-n}$ otherwise.

For emphasis, the operator $M_{z^{-1}}$ of division by $z$ is not well defined on $\widehat{\HH}$ since $0 \in \Omega$ and $\widehat{\HH}$ contains the constant functions.  Yet $M_{z^{-1}}$ is well defined as a map from  $\mbox{ran}(M_z) = \ker(M_z^\dagger)^\perp$ to $\widehat{\HH}$.  By the above computation, $M_z^\dagger$ is $M_{z^{-1}}$ on $\ker(M_z^\dagger)^\perp$.  Hence, $T^\dagger$ is $M_{z^{-1}}$  wherever the operator $M_{z^{-1}}$ is well defined, and $0$ otherwise.  This can be succinctly written as
\[
M_z^\dagger = M_{z^{-1}} Q_1
\]
where $Q_1$ is the projection onto $\ker(M_z^\dagger)^\perp$.   More generally for each $n$, we have that 
\[
{M_z^\dagger}^n = M_{z^{-n}} Q_n
\]
where $Q_n$ is the projection onto $\ker({M_z^\dagger}^n)^\perp$.

This model gives intuition into the structure of $\A_T$.   By Proposition \ref{alg simple},  $\mbox{\Alg}(M_z, M_z^\dagger)$ consists of operators of the form
\[
F + \sum_{k=0}^N a_k {M_z}^k + \sum_{l=1}^M b_l {M_z^\dagger}^l = F + \sum_{k=0}^N a_k M_{z^k} + \sum_{l=1}^M b_l M_{z^{-l}} Q_l
\]
where $F$ is a finite rank operator.  One could combine via linearity the ``analytic'' component of the above sum to get  
\[
F + M_{\sum_{k=0}^N a_k {z^k}} + \sum_{l=1}^M b_l M_{z^{-l}} Q_l.
\]
In some sense, the ``principal part'' $\sum_{l=1}^M b_l M_{z^{-l}} Q_l$ may also be combined into a single multiplication operator.  Unfortunately, this is not done as effortlessly. We do have that $Q_l \leq Q_k$ for all $k \leq l$.  Therefore, for all $\hat{f} \in \ker({T^\dagger}^M)^\perp$, the sum of the principal pieces combine into a single multiplication operator.  That is, 
\[
\left(\sum_{l=1}^M b_l M_{z^{-l}} Q_l \right) (\hat{f})(\lambda) = \sum_{l=1}^M b_l \frac{\hat{f}(\lambda)}{\lambda^l} = \left(M_{\sum_{l=1}^M b_l z^{-l}} \hat{f} \right)(\lambda)
\]
However, this fails on $\ker({T^\dagger}^M)$, as some operators in the principal part have kernels contained in $\ker({T^\dagger}^M)$. For example, if $\hat{f}$ is perpendicular to $\ker({T^\dagger}^L)$ but not perpendicular to $\ker({T^\dagger}^{L+1})$, then
\[
\left(\sum_{l=1}^M b_l M_{z^{-l}} Q_l \right) (\hat{f})(\lambda) = \sum_{l = 1}^L b_l \frac{\hat{f}(\lambda)}{\lambda^l} = \left(M_{\sum_{l=1}^L b_l z^{-l}} \hat{f} \right)(\lambda).
\]

This discussion demonstrates that we have a canonical analytic model to represent $\A_T$.   It is the norm limit of finite rank operators plus multiplication operators that have ``Laurent'' polynomials as symbols. 

\begin{heuristic}
	If $T$ is a natural analytic left invertible operator, then the algebra $\A_T$ is compact perturbations of multiplication operators whose symbols are Laurent series centered at zero.  
\end{heuristic}

In this section, we have shown that $T = M_z$ on a RKHS of analytic functions. To some extent, a converse statement is true as well. In \cite{Richter}, Richter shows if $T$ is $M_z$ on a reproducing kernel Hilbert space of analytic functions, then under suitable assumptions, $T$ is an analytic left invertible operator.  

%%%%%%%%%%%%%%%%%%%%%%%%%%%%%%%%%%%%%%%%%%%%%%%%%%%%%%%%%%
%%%%%%%%%%%%%%%%%%%%%%%%%%%%%%%%%%%%%%%%%%%%%%%%%%%%%%%%%%
%%%%%%%%%%%%%%%%%%%%%%%%%%%%%%%%%%%%%%%%%%%%%%%%%%%%%%%%%%
%%%%%%%%%%%%%%%%%%%%%%%%%%%%%%%%%%%%%%%%%%%%%%%%%%%%%%%%%%
\subsection{Reduction of Index - Strongly Irreducible Operators}
%%%%%%%%%%%%%%%%%%%%%%%%%%%%%%%%%%%%%%%%%%%%%%%%%%%%%%%%%%
%%%%%%%%%%%%%%%%%%%%%%%%%%%%%%%%%%%%%%%%%%%%%%%%%%%%%%%%%%
%%%%%%%%%%%%%%%%%%%%%%%%%%%%%%%%%%%%%%%%%%%%%%%%%%%%%%%%%%
%%%%%%%%%%%%%%%%%%%%%%%%%%%%%%%%%%%%%%%%%%%%%%%%%%%%%%%%%%

Suppose that  $T$ is an analytic (pure) isometry with Fredholm index $-n$ for $n \geq 2$.  Then $T$ can be decomposed as a direct sum of pure isometries $T_i$ each with Fredholm index -1.  This decomposition is clearly unique up to unitary equivalence. A similar, though much weaker, statement is true for general analytic left invertible operators. We require some terminology.

\begin{defn}[\cite{Jiang}]
	An operator $R \in \BH$ is \textbf{strongly irreducible} if there is no non-trivial idempotent in $\{R\}'$, the commutant of $R$. Equivalently, $R$ is strongly irreducible if $XRX^{-1}$ is an irreducible operator for every invertible operator $X$. We denote the set of all strongly irreducible operators over $\HH$ by $(SI)$. 
\end{defn}
Strong irreducibility is  a similarity invariant. Moreover, it follows by definition that $R \in (SI)$ if and only if $R^* \in (SI)$. Strongly irreducible operators play an important role in single operator theory.  They serve a role equivalent to the Jordan blocks in the infinite dimensional setting. 

In finite dimensions, the strongly irreducible operators are easily seen to be Jordan matrices. If $A \in M_n$, the Jordan canonical forms theorem states that $A$ is similar to a direct sum of Jordan blocks.  This decomposition is unique, up to the ordering of the blocks.  If $\sigma(A) = \{\lambda_i\}_{i=1}^n$, then we write
\[
A \sim \bigoplus_{i=1}^l J_{k_i} (\lambda_i)^{(m_i)}
\]
where the superscript $(m_i)$ denotes the orthogonal direct sum of $m_i$ copies of the Jordan block $J_{k_i} (\lambda_i)$.  In other words, the Jordan decomposition theorem states that, up to similarity, each matrix has a unique decomposition as a direct sum of strongly irreducible operators. 

The same statement translates into the infinite dimensional setting.  To help make this more precise, we have the following definition:

\begin{defn}[\cite{Jiang}]
	A sequence $\{E_j\}_{j=1}^l$, $1 \leq l \leq \infty$ of non-zero idempotents on $\HH$ is called a \textbf{spectral family} if 
	
	\begin{enumerate}[i.]
		\item there exists an invertible operator $X \in \BH$ such that $\{X E_j X^{-1}\}$ are pairwise orthogonal projections
		
		\item $\sum_{j=1}^l E_j = I$.
	\end{enumerate}
	\noindent Furthermore, if $R \in \BH$, then the spectral family is a \textbf{strongly irreducible decomposition of $R$} if
	
	\begin{enumerate}[i.]
		\setcounter{enumi}{2}
		\item $E_j R = R E_j$ for all $j$
		\item $R\mid{\mbox{ran}(E_j)} \in (SI)$. 
	\end{enumerate}
	
\end{defn}

In other words, $R$ has a strongly irreducible decomposition if $R$ is the topological direct sum strongly irreducible operators.  Equivalently, $R$ is similar to the orthogonal direct sum of strongly irreducible operators.  We denote this by $R \sim \oplus_{j=1}^l R_j$.

In finite dimensions,  Jordan canonical forms force each matrix to have a unique SI decomposition up to similarity. This is not the case for operators in $\BH$.  Not every operator in $\BH$ has a strongly irreducible decomposition. Moreover, even if an operator has a strongly irreducible decomposition, it may not be unique \cite{Jiang3}. Therefore, we make the following definition:

\begin{defn}
	Let $R \in \BH$, and $\mathcal{E} =  \{E_j\}_{j=1}^{l_1}$ and $\mathcal{E}' =  \{E_j'\}_{j=1}^{l_2}$ be two strongly irreducible decompositions of $R$.  We say  $\mathcal{E}$ and $\mathcal{E}'$ are \textbf{similar} if
	
	\begin{enumerate}[i.]
		\item $l_1 = l_2 = l$
		\item there exists an invertible operator $X \in \{R\}'$, the commutant of $R$, such that $X E_j X^{-1} = E_j'$ for all $1 \leq j \leq l$.
	\end{enumerate}
	
	\noindent	If $R$ has a strongly irreducible decomposition, we say that $R$ has a \textbf{unique strongly irreducible decomposition up to similarity} if any two of the decompositions are similar.
\end{defn}

There is an extensive amount of work relating strongly irreducible decompositions of operators to K-theory \cite{Cao}, \cite{Jiang}, \cite{Jiang2}, \cite{Jiang3}.  We will mention some of these results in in a later section.  Of particular interest to us in the present are the following deep results due to Y. Cao, J. Fang and C. Jiang:

\begin{thm}[\cite{Jiang3} - Theorem 5.5.12]
	\label{SI}
	Each operator in $S \in B_1(\Omega)$ is strongly irreducible. Moreover for any $n$, if $R \in B_n(\Omega)$, then $R$ has a unique SI decomposition up to similarity. Furthermore,  $R \sim \oplus_{j=1}^m R_j$ where $R_j \in (SI) \cap B_{n_j}(\Omega)$ and $\sum_{j=1}^m n_j = n$.
	
\end{thm}

\begin{cor}
	Let $T$ be an analytic left invertible operator with $\mbox{\ind}(T) = -n$ for some $1 \leq n < \infty$.  Then $T \sim \oplus_{j=1}^m T_j$ where $T_j$ are analytic, $\sum_{j=1}^m \mbox{\ind}(T_j) = -n$ and $T_j \in (SI)$.
\end{cor}

%\begin{proof}
%	By Theorem \ref{CD}, $T^* \in B_n(\Omega)$ for some disc $\Omega$ centered at the origin.  Therefore by Theorem \ref{SI}, $T^* \sim \oplus_{j=1}^m R_j$ where each  $R_j \in (SI) \cap B_{n_j}(\Omega)$ and $\sum_{j=1}^m n_j = n$.  By another application of Theorem \ref{CD}, $T_j:= R_j^*$ are analytic left invertibles that satisfy $\sum_{j=1}^m \mbox{ind}(T_j) = -n$.  Since $R_j \in (SI)$ and strong irreducibility is preserved under taking adjoints, $T_j \in (SI)$.  \qedhere
%\end{proof}

%Spacing issue?!?!?!?!

Theorem \ref{SI} states that operators in the Cowen-Douglas class have a decomposition analogous to the Jordan canonical forms for matrices.  Without loss of generality, we may assume that if $R \in B_n(\Omega)$, then $R = \oplus_{j=1}^m R_j$ where  $R_j \in (SI) \cap B_{n_j}(\Omega)$ where $\sum_{j=1}^m n_j = n$. This decomposition  suggests that in order to understand $\A_T$, we should first study the natural analytic left invertible operators that are strongly irreducible.  In particular, we should study the analytic left invertible operators with Fredholm index $-1$. 

In the isometric case, $T^* \in B_n(\Omega)$ decomposes to a direct sum of $n$ strongly irreducible operators in $B_1(\Omega)$.  Equivalently, pure isometric operators with $\mbox{ind}(T) = -n$ decompose into $n$ ``Jordan blocks'' of size $1$.  This turns out to not be the case in general. Notice that if $R \in B_n(\Omega) \cap (SI)$, then it cannot be further decomposed as a direct sum.  Indeed, suppose to the contrary that $R \in B_n(\Omega) \cap (SI)$ and $R \sim \oplus_{k=1}^n R_k$ with $R_k \in B_1(\Omega)$.  By Theorem \ref{SI}, each operator in $B_1(\Omega)$ is strongly irreducible.  Hence, $R$ would have two strongly irreducible decompositions that are dissimilar.  But Theorem \ref{SI} states that all Cowen-Douglas operators have a unique SI decomposition up to similarity, contradicting the assumption that $R \in B_n(\Omega) \cap (SI)$ and $R \sim \oplus_{k=1}^n R_k$ .

 For each $n$, one can construct Toeplitz operators $T$ over the Soblev space $W^{2,2}(\Omega)$ such that $T^* \in B_n(\Omega) \cap (SI)$. See \cite{DeSantis} for details. Thus, it is not always possible to decompose $T$ as a direct sum of left invertibles with Fredholm index $-1$. Nevertheless, Cowen-Douglas operators of rank $n$ take the form of triangular operators of size $n$:

\begin{thm}[\cite{Jiang} - Theorem 1.49]
	\label{CD decomp}
	Let $R \in B_n(\Omega)$ for $1 \leq n < \infty$.  Then there exists $n$ operators $R_1, \dots R_n$ such that $R_i \in B_1(\Omega)$ and
	\[
	R = \begin{pmatrix} 
	R_1 & * &* &*\\
	& R_2 &* &* \\
	& & \ddots & \vdots\\
	& &  & R_n
	\end{pmatrix}
	\]
	with respect to some decomposition $\HH = \oplus_{i=1}^n \HH_i$.
\end{thm}

\begin{cor}
	\label{reduction}
	If $T$ is an analytic left invertible with $\mbox{\ind}(T) = -n$ for $1 \leq n < \infty$, then there exists $n$ analytic left invertibles $T_1, \dots T_n$ such that $\mbox{\ind}(T_i) = -1$ and
	\begin{equation}
	\label{decomp}
	T = \begin{pmatrix}
	T_1 &  &  & \\
	* & T_2 &  & \\
	\vdots & \vdots & \ddots & \\
	* & * & \dots & T_n 
	\end{pmatrix}
	\end{equation}
	with respect to some decomposition $\HH = \oplus_{i=1}^n \HH_i$.
\end{cor}

Corollary \ref{reduction}  further emphasizes the need to analyze analytic left invertible operators with $\mbox{\ind}(T) = -1$.  We showed above that we can always decompose $T$ into a direct sum of strongly irreducible pieces.  The strongly irreducible blocks have the form of lower triangular operators.  If $T$ is decomposed as in Corollary \ref{reduction}, then $T_n = T\mid_{\HH_n}$ and $T_n$ is an analytic left invertible operator with $\mbox{\ind}(T_n) = -1$.  If we are to gain any insight into a general $\A_T$, it is mandatory to understand the index $-1$ case first.  This analysis will be taken up next section.

%%%%%%%%%%%%%%%%%%%%%%%%%%%%%%%%%%%%%%%%%%%%%%%%%%%%%%%%%%
%%%%%%%%%%%%%%%%%%%%%%%%%%%%%%%%%%%%%%%%%%%%%%%%%%%%%%%%%%
%%%%%%%%%%%%%%%%%%%%%%%%%%%%%%%%%%%%%%%%%%%%%%%%%%%%%%%%%%
%%%%%%%%%%%%%%%%%%%%%%%%%%%%%%%%%%%%%%%%%%%%%%%%%%%%%%%%%%
%%%%%%%%%%%%%%%%%%%%%%%%%%%%%%%%%%%%%%%%%%%%%%%%%%%%%%%%%%
%%%%%%%%%%%%%%%%%%%%%%%%%%%%%%%%%%%%%%%%%%%%%%%%%%%%%%%%%%
%%%%%%%%%%%%%%%%%%%%%%%%%%%%%%%%%%%%%%%%%%%%%%%%%%%%%%%%%%
%%%%%%%%%%%%%%%%%%%%%%%%%%%%%%%%%%%%%%%%%%%%%%%%%%%%%%%%%%
\section{The Algebra $\A_T$ for $\mbox{ind}(T)=-1$}
%%%%%%%%%%%%%%%%%%%%%%%%%%%%%%%%%%%%%%%%%%%%%%%%%%%%%%%%%%
%%%%%%%%%%%%%%%%%%%%%%%%%%%%%%%%%%%%%%%%%%%%%%%%%%%%%%%%%%
%%%%%%%%%%%%%%%%%%%%%%%%%%%%%%%%%%%%%%%%%%%%%%%%%%%%%%%%%%
%%%%%%%%%%%%%%%%%%%%%%%%%%%%%%%%%%%%%%%%%%%%%%%%%%%%%%%%%%
%%%%%%%%%%%%%%%%%%%%%%%%%%%%%%%%%%%%%%%%%%%%%%%%%%%%%%%%%%
%%%%%%%%%%%%%%%%%%%%%%%%%%%%%%%%%%%%%%%%%%%%%%%%%%%%%%%%%%
%%%%%%%%%%%%%%%%%%%%%%%%%%%%%%%%%%%%%%%%%%%%%%%%%%%%%%%%%%
%%%%%%%%%%%%%%%%%%%%%%%%%%%%%%%%%%%%%%%%%%%%%%%%%%%%%%%%%%

The preceding sections showed that, in general, we cannot reduce to the assumptions analytic or $\mbox{\ind}(T) = -1$ as we could in the isometric case.  As remarked, $T$ cannot be decomposed as a direct sum of an analytic operator and an invertible operator.  Furthermore,  even if an operator is analytic, it cannot be reduced to the index $-1$ case. Nevertheless, there is a summand on which $T$ will be analytic.   Similar statements may be made about strong irreducibility and the Fredholm index.  Under the assumption of analytic, Theorem \ref{CD} implies that $T^*$ is Cowen-Douglas.  Corollary \ref{reduction} tells us that, in this case, $T$ may be written as a triangular operator where each element on the diagonal is an analytic left invertible of index $-1$.  

Although we cannot reduce to the case of analytic or index $-1$, the epistemological viewpoint of the author is that an important first step in understanding $\A_T$ is simplifying to this case. We therefore make the following minimality assumptions on $T$ for the remainder of this section:

\begin{assumption}
 Henceforth, our left invertible operators will satisfy
	
	\begin{enumerate}[i.]
		\item The Fredholm index: $\mbox{\ind}(T) = -1$
		\item Analytic:  $\bigcap T^n \HH = 0$
	\end{enumerate}
	
\end{assumption}

If $T$ is an analytic isometry with $\mbox{ind}(T) = -1$, we can represent $T$ as $M_z$ on $H^2(\mathbb{T})$.  This yields an elegant representation for $C^*(T)$.  The analyticity ensures that the basis associated to $M_z$, the orthonormal basis $z^n$, spans the  Hilbert space.  The Fredholm index guarantees that $\mathcal{T}$ will be an irreducible C*-algebra, which contains a compact $I-T T^*$, and therefore all the compacts. Furthermore, one discovers that each element of $\mathcal{T}$ may be uniquely written as $T_f + K$ for some $f \in C(\mathbb{T})$ and $K \in \KH$. 

The general case is similar.  That is, if $T$ is an analytic, left invertible operator with Fredholm index $-1$, then $\A_T$ contains the compact operators.  This will allow us to determine the isomorphism classes of $\A_T$.  

It is worth remarking that since $\A_T$ is a concrete operator algebra, it belongs to many reasonable categories.  A priori, it is not clear which choice of morphism one should consider (bounded, completely bounded, etc.). Fortunately, all reasonable choices are equivalent. It will be shown that two such algebras are boundedly isomorphic if and only if the isomorphism is implemented by an invertible. This will bring us to analyze the similarity orbit of $T$.  For Cowen-Douglas operators, the similarity orbit has been extensively studied.  We will leverage these results into our analysis of the study of $\A_T$.

%%%%%%%%%%%%%%%%%%%%%%%%%%%%%%%%%%%%%%%%%%%%%%%%%%%%%%%%%%
%%%%%%%%%%%%%%%%%%%%%%%%%%%%%%%%%%%%%%%%%%%%%%%%%%%%%%%%%%
%%%%%%%%%%%%%%%%%%%%%%%%%%%%%%%%%%%%%%%%%%%%%%%%%%%%%%%%%%
%%%%%%%%%%%%%%%%%%%%%%%%%%%%%%%%%%%%%%%%%%%%%%%%%%%%%%%%%%
\subsection{The Compact Operators}
%%%%%%%%%%%%%%%%%%%%%%%%%%%%%%%%%%%%%%%%%%%%%%%%%%%%%%%%%%
%%%%%%%%%%%%%%%%%%%%%%%%%%%%%%%%%%%%%%%%%%%%%%%%%%%%%%%%%%
%%%%%%%%%%%%%%%%%%%%%%%%%%%%%%%%%%%%%%%%%%%%%%%%%%%%%%%%%%
%%%%%%%%%%%%%%%%%%%%%%%%%%%%%%%%%%%%%%%%%%%%%%%%%%%%%%%%%%
In this section, we show that  if $T$ is analytic left invertible with $\mbox{\ind}(T) = -1$, then  $\A_T$ contains the compact operators. Our approach is to show that, more generally $\overline{\mbox{\Alg}}(T,L)$ contains the compact operators for any left inverse $T$ and left inverse $L$.  This will allow us to conclude that $\overline{\mbox{\Alg}}(T,L) = \A_T$ for any left inverse $L$.  First, let us establish some notation.

Fix a left inverse $L$  of $T$. We set $F_{0,0} = I - T T^\dagger$. That is, $F_{0,0}$ is the projection onto $\ker(T^\dagger)$.  We define
\[
F_{n,m,L}:= T^{n} (I - T T^\dagger) {L}^{m}
\]
for each $n,m \in \mathbb{Z}_{\geq 0}$. For $x,y,z \in \HH$ we use $\theta_{x,y}$  to denote the rank one operator $z \mapsto \langle z, y \rangle x$.

Recall the Schauder basis and dual basis associated to $T$ and $L$. Notice that since $\mbox{\ind}(T) = -1$, we have a simplified notation. Let $x_0 \in \ker(T^*)$ be a unit vector, so $\mbox{\Span}\{x_0\} = \ker(T^*)$.  Denote the Schauder basis of $T$ and dual basis $T$ (with respect to $L$) via $x_n:= T^n x_0$ and $x_n':= (L^*)^n x_0$. Then by definition, $I-TT^\dagger$ is the projection $\theta_{x_0,x_0}$. So for each $n,m$ and $x \in \HH$,
\[
F_{n,m,L}(x) = T^{n} (I - T T^\dagger) {L}^{m}(x) = T^{n}(\langle {L}^{m}(x), x_0 \rangle x_0) = \langle x, x_m' \rangle x_n.
\]
That is, $F_{n,m,L}$ is the rank one operator $\theta_{x_n,x_m'}$. Let
\[
\mathscr{K}_L := \overline{\mbox{\Span}}\{F_{n,m,L}\}_{n,m \geq 1}.
\]
 Recall from Proposition \ref{commutator} that if $L = T^\dagger$, then $\mathscr{K}_L = \mathscr{K}_T = \mathscr{C}$, the commutator ideal.   As $F_{n,m,L} \in \mbox{\Alg}(T,L)$,  $\mathscr{K}_L \subset \overline{\mbox{\Alg}}(T,L)$.  Furthermore, the $F_{n,m,L}$ are rank one operators for each $n,m$; and so $\mathscr{K}_L \subset \mathscr{K}(\mathscr{H})$. Our previous work on Schauder bases allows us to conclude that $\mathscr{K}_L = \KH$.

\begin{thm}
	\label{compact}
	Let $T \in \BH$ be an analytic, left invertible with $\mbox{\ind}(T) = -1$, and $L$ be a left inverse of $T$.  Then $\KH = \KK_L$.  Thus,  $\overline{\mbox{\Alg}}(T,L)$ contains the algebra of compact operators $\KH$.  
\end{thm}

\begin{proof}
	Let $y,z \in \HH$.  Since $\overline{\mbox{\Span}}\{x_n\} = \HH = \overline{\mbox{\Span}}\{x_n'\}$, there exists a sequence of  sums in $x_n$ and $x_n'$ converging to $y$ and $z$ respectively.  It follows that the rank one operator $\theta_{y,z}$ is a norm limit of the span of the $\{F_{n,m,L}\}$ by simple estimates. Thus, $\mathscr{K}_L$ contains all the rank one operators.  Since $\KK_L$ is norm-closed by definition, $\KK_L \supset \KH$. Since $\mathscr{K}_L \subset \mathscr{K}(\mathscr{H})$, we have $\KK_L = \KH$. 
\end{proof}

\begin{cor}
	\label{any left}
	Let $T \in \BH$ be left invertible (analytic with $\mbox{\ind}(T) = -1$), and $L$ be a left inverse of $T$.  Then  $\A_T = \overline{\mbox{\Alg}}(T,L)$.
\end{cor}

\begin{proof}
	By Proposition \ref{left_facts}, each left inverse $L$ of $T$ has the form 
	\[
	L= T^\dagger + A(I-T T^\dagger)
	\]
	for some $A \in \BH$.	Thus, each left inverse of $T$ differs from $T^\dagger$ by a compact operator.  By Theorem \ref{compact}, $\overline{\mbox{\Alg}}(T,L)$ contains $\KH$, and therefore $T^\dagger$.  So $\overline{\mbox{\Alg}}(T,L) \subseteq \A_T$.  Reversing the argument, $\overline{\mbox{\Alg}}(T,L) = \A_T$.
\end{proof}

Recall that an ideal $\KK$ of a Banach Algebra $\A$ is said to be \textit{essential} if it has non-trivial intersection with  all non-zero ideals of $\A$.  Alternatively, if $A \in \A$ and $A \KK =0$, then $A = 0$. In the next section, we investigate the morphisms between algebras of the form $\A_T$.  An important result required in subsequent analysis is the following: 
\begin{prop}
	\label{essential}
	The compact operators $\KH$ are an essential ideal of $\A_T$.  In fact, $\KH$ is contained in any closed ideal of $\A_T$.
\end{prop}

\begin{proof}
	Let $\J$ be a non-zero closed two sided ideal of $\A_T$, and $A \in \J$ be non-zero.  Then there is some $x \in \HH$ such that $\|A x \|=1$.  Fix $y \in \HH$, and let $B := \theta_{y,A(x)}$.  Then $B(A(x)) = y$.  Thus for all $h \in \HH$, we have
	\[
	B A \theta_{x,x} A^* B^*(h) = B A \left( \langle h, BA (x) \rangle x \right) = \langle h, y \rangle y = \theta_{y,y} (h).
	\]
	Since $\KH \subset \A_T$, it follows that the rank one operators $B$ and $\theta_{x,x} A^* B^*$ are in $\A_T$. Since $A \in \J$ and $\J$ is an ideal, we must have that $\theta_{y,y}$ is inside of $\J$.  Thus for any $w,z \in \HH$, $\theta_{w,z} = \theta_{w,y} \theta_{y,y} \theta_{y,z}$ is in $\J$, so $\J$ contains all the finite rank operators, and thus contains $\KH$. 
\end{proof}

%%%%%%%%%%%%%%%%%%%%%%%%%%%%%%%%%%%%%%%%%%%%%%%%%%%%%%%%%%
%%%%%%%%%%%%%%%%%%%%%%%%%%%%%%%%%%%%%%%%%%%%%%%%%%%%%%%%%%
%%%%%%%%%%%%%%%%%%%%%%%%%%%%%%%%%%%%%%%%%%%%%%%%%%%%%%%%%%
%%%%%%%%%%%%%%%%%%%%%%%%%%%%%%%%%%%%%%%%%%%%%%%%%%%%%%%%%%
\subsection{Isomorphisms of $\A_T$}
%%%%%%%%%%%%%%%%%%%%%%%%%%%%%%%%%%%%%%%%%%%%%%%%%%%%%%%%%%
%%%%%%%%%%%%%%%%%%%%%%%%%%%%%%%%%%%%%%%%%%%%%%%%%%%%%%%%%%
%%%%%%%%%%%%%%%%%%%%%%%%%%%%%%%%%%%%%%%%%%%%%%%%%%%%%%%%%%
%%%%%%%%%%%%%%%%%%%%%%%%%%%%%%%%%%%%%%%%%%%%%%%%%%%%%%%%%%

Now that we have established that the compact operators $\KH \subseteq \A_T$ as a minimal ideal, we may identify the isomorphism classes of $\A_T$.  We will show that if $T_1$ and $T_2$ are two analytic left invertible operators with Fredholm index $-1$, then $\A_{T_1}$ is boundedly isomorphic to $\A_{T_2}$ if and only if the algebras are similar. This will be done by looking at how the bounded isomorphism behaves on the compact operators.

An interesting fact about bounded homomorphisms of C*-algebras is that they necessarily have closed range.  Indeed, we have the following observation due to Pitts:

\begin{thm}[\cite{Pitts} - Theorem 2.6]
	Suppose $\A$ is a C*-algebra and $\phi:\A \rightarrow \BH$ is a bounded homomorphism.  Let $\J = \ker \phi$.  Then there exists a real number $k > 0$ such that for each $n \in \mathbb{N}$, and $R \in M_n(\A)$,
	\[
	k \mbox{dist}(R, M_n(\J)) \leq \|\phi_n(R)\|.
	\]
\end{thm}

\begin{cor}
	If $\phi: \KH \rightarrow \BH$ is a bounded monomorphism, then there exists a real number $k$ such that
	\[
	k \|R\| \leq \|\phi(R)\|.
	\]
	That is, $\phi$ has closed range.
\end{cor}

% Given an invertible operator $V \in \BH$, we define $\Adv:\BH \rightarrow \BH$ via $\Adv(T) = V T V^{-1}$. As previously mentioned, to fully analyze $\A_T$, we need to determine which category we are working in.  On the one hand, we can view $\A_T$ as an operator algebra, with our morphisms being completely bounded homomorphisms.  On the other hand, we may want to simply view $\A_T$ as a Banach algebra, where the morphisms are bounded homomorphisms. Fortunately, Theorem \ref{compact} forces the monomorphisms of these two categories to coincide:
We are now in a position to prove the main theorem of the paper.  Given an invertible operator $V \in \BH$, we define $\Adv:\BH \rightarrow \BH$ via $\Adv(T) = V T V^{-1}$.

\begin{thmy}
	\label{similar}
	Let $T_i$, $i=1,2$ be left invertibles (analytic with $\mbox{\ind}(T_i) = -1$) and $\A_i = \A_{T_i}$.  Suppose that  $\phi:\A_1 \rightarrow \A_2$ is a bounded isomorphism.  Then $\phi = \Adv$ for some invertible $V \in \BH$. 
\end{thmy}

\begin{proof}
	Let $\phi:\A_1 \rightarrow \A_2$ be a bounded isomorphism.  A brief outline of the proof is as follows.  We first show that $\phi\mid_{\KH}$ is similar to a *-automorphism of $\KH$.  It is well known that all *-automorphisms of $\KH$ have the form $\Adu$ for some unitary operator $U$.  We then use the fact that $\phi$ restricted to an essential ideal has the form $\Adv$ to conclude that it must be equal to $\Adv$ on all of $\A_1$. The details are as follows.  
	
	Note that $\phi\mid_{\KH}: \KH \rightarrow \A_2 \subset \BH$ is a bounded representation of the compact operators. It can be shown that every bounded representation of the compact operators is similar to a *-representation (more generally, every bounded representation of a nuclear C*-algebra is similar to a *-representation \cite{Christensen}).  Let $W \in \BH$ be the invertible that conjugates $\phi\mid_{\KH}$ to a *-representation $\psi$. That is, $\phi(u) = W \psi(u) W^{-1}$ for every $u \in \KH$.  
	
	Now let us consider the $*-$representation $\psi$.  Note that $\psi:\KH \rightarrow W^{-1} \A_2 W$. The map $\mbox{Ad}_{W^{-1}}: \A_2 \rightarrow W^{-1} \A_2 W$ carries $\KH$ to $\KH$.  Since every ideal of $W^{-1} \A_2 W$ has the form $W^{-1} \J W $ for $\J$ an ideal of $\A_2$, it follows that $\KH$ is minimal in $W^{-1} \A_2 W$.  Therefore, we must have that $\KH \subseteq \psi(\KH)$. 
	
	Now, $\KH$ is equal to the closed span of the rank one  projections on $\HH$.  As a result, if we can show that each rank one projection $p$ gets sent to another rank one projection under $\psi$, then $\psi(\KH) \subset \KH$, yielding equality. 
	
	To this end, let $p$ be a rank one projection, and $p' = \psi(p)$. If $p'$ is not rank one, then there exists a non-zero projection $q'$ properly contained under $p'$.  Since $\psi(\KH)$ contains $\KH$, there exists a projection $q \in \KH$ such that $\psi(q) = q'$.  Regarding $\psi$ mapping from $\KH$ to $\psi(\KH)$, $\psi$ is a *-isomorphism and hence invertible.  $\psi^{-1}$ is of course also a *-isomorphism, and therefore a positive map. Hence, if $q' < p'$, then $q < p$ by positivity of $\psi^{-1}$. This is absurd, since $p$ was rank one. Thus, $\psi(\KH) \subset \KH$, so that $\KH = \psi(\KH)$. 
	
	What we have just shown is that $\phi\mid_{\KH}$ is similar to a *-automorphism $\psi$ of $\KH$.  Every *-automorphism of $\KH$ is of the form $\mbox{Ad}_U$ for some unitary operator $U$.  Hence, we have that 
	\[
	\phi\mid_{\KH} = \Adw \psi = \Adw \mbox{Ad}_U = \mbox{Ad}_{V}
	\]
	where $V = UW$. We now show that $\phi = \Adv$.  To do this, first  note that for all $A \in \A_1$ and $K \in \KH$, 
	\[
	\phi(A) \phi(K) = \phi(AK) = \psi(AK) =  \Adv(AK) = \Adv(A) \Adv(K) = \Adv(A) \phi(K)
	\]
	So it follows that 
	\[
	(\phi(A) - \Adv(A)) \Adv(K) = 0
	\]
	for each $K \in \KH$. Cycling over all $K \in \KH$, we see that 
	\[
	(\phi(A) - \Adv(A) )\KH = 0.
	\]
	Since $\KH$ is essential in $\A_2$, we have that $\phi(A) = \Adv(A)$. 
\end{proof}

Theorem \ref{similar} is a harsh rigidity statement about classification.  Indeed, $\A_1$ is boundedly isomorphic to $\A_2$ if and only if the algebras are similar.  Consequently, if we wish to delineate these operator algebras into isomorphism classes, we need to understand the similarity orbit of left invertible operators. We define the following notation for the similarity orbit:
\[
\mathcal{S}(T):= \{V T V^{-1}: V \in \BH \mbox{ is invertible} \}.
\]

In classifying the algebra $\A_T$,  we do not need to keep track of the similarity orbit of the Moore-Penrose inverse.  Indeed, suppose $T$ is left invertible with Moore-Penrose inverse $T^\dagger$,  $V$ is an invertible operator, and $T_2:= V T V^{-1}$.  Then $L_2:= V T^{\dagger} V^{-1}$ is a left inverse of $T_2$.  By Corollary \ref{any left}, $\overline{\mbox{\Alg}}(T_2,L_2) = \A_{T_2}$. Therefore to identify the isomorphism class of $\A_T$, we may disregard $\mathcal{S}(T^\dagger)$.  Hence, we pose the following question:

\begin{question}
	If $T$ is left invertible (analytic, $\mbox{\ind}(T) = -1$), what is $\mathcal{S}(T)$?
\end{question}

In general, it is impossible to completely classify the similarity orbit of an operator.  However, analytic left invertible operators have added structure that aid in this analysis.  By Theorem \ref{CD}, if $T$ is analytic,  $T^* \in B_n(\Omega)$ for a  disc $\Omega$ centered at the origin.   Clearly if we could identify $\mathcal{S}(T^*)$, then we would know $\mathcal{S}(T)$.  Fortunately, similarity orbits of Cowen-Douglas operators have been extensively studied  \cite{Cowen2} \cite{Curto} \cite{Jiang2} \cite{Li}  \cite{Zhu}.  The similarity orbit of Cowen-Douglas operators can be completely described by K-theoretic means.  We will highlight these results in the next section.

While the question of addressing the similarity orbit is paramount to a complete classification of our algebras $\A_T$, it is  not sufficient.  Explicitly, suppose $T_1$ and $T_2$ are left invertible operators (analytic, $\mbox{\ind}(T) = -1$) with $\A_1$ and $\A_2$ isomorphic.  Let $V$ be the invertible that implements the isomorphism between $\A_1$ and $\A_2$, and let $T_3:= V T_1 V^{-1}$ and $L_3:= V T_1^{\dagger} V^{-1}$. Notice $L_3$ is a left inverse of $T_3$ and that $\overline{\mbox{\Alg}}(T_3, L_3) = \A_2$.  By Corollary \ref{any left}, $\A_3 = \A_2$.  

One would therefore be tempted to reduce to the case where $T_2 = T_3 = \Adv(T_1)$. However, it turns out that not every left invertible $S \in \A_T$ will satisfy $\A_{S} = \A_T$.  Consider the following example:

\begin{example}
	We will  construct a left invertible operator $T$ inside the Toeplitz algebra $\mathcal{T}$ such that $\A_T \neq \mathcal{T}$.  Consider the Hardy space $H^2(\mathbb{T})$.  Let $\phi_0 \in C(\mathbb{T})$ be given by 
	\[
	\phi_0(z) := \mbox{exp}\left( \frac{\pi i}{2} (z-1)z \right)
	\]
	for all $z \in \mathbb{T}$.  Then $\phi_0(1) = 1$ and $\phi_0(-1) = -1$.  Let 
	\[
	\epsilon_n(z) = z^n. 
	\]
	 Define $\phi := M_{\epsilon_1} \phi_0$.  Then $\phi$ satisfies $\phi(1) = \phi(-1) = 1$. 
	 
	 %Recall the following facts about invertible functions on $C(\mathbb{T})$ and their associated Toeplitz operators:
	
% 	\begin{thm}[\cite{Murphy} Lem. 3.5.14, Thm. 3.5.15]
% 		\label{ThmMurphy}
% 		Let $\phi \in C(\mathbb{T})$ be invertible. Then
		
% 		\begin{enumerate}[i.]
% 			\item There exists a unique integer $n$ such that $\phi = \epsilon_n e^\psi$ some $\psi \in C(\mathbb{T})$
% 			\item If  $\phi = \epsilon_n e^\psi$, then the winding number is $n$
% 			\item We have $\mbox{\ind}(T_\phi) = $ negative the winding number of $\phi$
% 			\item $T_\phi$ is invertible if and only if the winding number is zero if and only if $\phi = e^\psi$ some $\psi \in C(\mathbb{T})$
% 		\end{enumerate}

% 	\end{thm}
	The winding number of $\phi$ is $1$, so $\mbox{\ind}(T_\phi) = -1$.  Since both $\epsilon_1$ and $\phi_0$ belong to $ H^\infty(\mathbb{T})$  we have that $T_{\epsilon_1}$ and $T_{\phi_0}$ commute, so the Toeplitz operator $T_\phi$ factors:
	\[
	T_\phi =   T_{\epsilon_1 \phi_0} = T_{\epsilon_1} T_{\phi_0}.
	\]
	As $z(z-1)$ is continuous on $\mathbb{T}$, $T_{\phi_0}$ is invertible.  The point-wise inverse of $\phi_0$ is also continuous on $\mathbb{T}$. Therefore, the Toeplitz operator $T_\phi$ is left invertible with left inverse 
	\[
	L= {T_{\phi_0}}^{-1} T_{\epsilon_1}^*  = T_{{\phi_0}^{-1}} T_{\epsilon_1}^* \in \mathcal{T}.
	\]
	Moreover, since $T_{\epsilon_1}$ and $T_{\phi_0}$ commute, we have $(T_\phi)^n = T_{\epsilon_n} T_{{\phi_0}^n}$.  Since  $T_{{\phi_0}^n}$ is invertible, $T_{{\phi_0}^n} H^2(\mathbb{T}) = H^2(\mathbb{T})$.  Consequently,
	\[
	\bigcap {T_\phi}^n H^2(\mathbb{T}) = \bigcap T_{\epsilon_n} H^2(\mathbb{T}) = 0
	\]
	so $T_\phi$ is analytic.  Recall that  $\A_T\subset C^*(T)$ for any left invertible $T$. We  remark that $C^*(T_\phi) \neq \mathcal{T}$.  This follows from the following result due to Coburn:
	
	\begin{lemma}[\cite{Coburn2} Cor. 6.3]
		If $\phi$ is in the disc algebra, then $C^*(T_\phi) = \mathcal{T}$ if and only if $\phi$ is injective.
	\end{lemma}
	It is shown in \cite{Coburn2} that  $C^*(T_\phi)/\KH$ is isomorphic to continuous functions on $\mathbb{T}/\sim$, where $\sim$ is an equivalence relation identifying all points $z,w \in \mathbb{T}$ such that $\phi(z) = \phi(w)$.  Since $\phi(1) = \phi(-1)$, it follows by the above lemma that $\A_T \subseteq C^*(T_\phi) \neq \mathcal{T}$. This concludes our example.
	
\end{example}

What the above example demonstrates is that not every left invertible operator in $\A_T$ generates $\A_T$. Therefore, determining the similarity orbit is not sufficient to delineate the isomorphism classes of $\A_T$.    Concretely, suppose  $\A_1$ and $\A_2$ are generated by $T_1$ and $T_2$ respectively. To  determine if $\A_1$ is isomorphic to $\A_2$, it is not sufficient to verify that $\A_2$ possesses an operator $T_3$ similar to $T_1$. This would demonstrate that $\A_1$ is isomorphic to a subalgebra of $\A_2$.  If one wanted $\A_1$ to be isomorphic to $\A_2$,  it is necessary to show that $T_3$ also generates $\A_2$.  With this caveat emphasized, we spend the next section investigating the similarity orbit of our class of left invertible operators.

%%%%%%%%%%%%%%%%%%%%%%%%%%%%%%%%%%%%%%%%%%%%%%%%%%%%%%%%%%
%%%%%%%%%%%%%%%%%%%%%%%%%%%%%%%%%%%%%%%%%%%%%%%%%%%%%%%%%%
%%%%%%%%%%%%%%%%%%%%%%%%%%%%%%%%%%%%%%%%%%%%%%%%%%%%%%%%%%
%%%%%%%%%%%%%%%%%%%%%%%%%%%%%%%%%%%%%%%%%%%%%%%%%%%%%%%%%%
\subsection{The Similarity Orbit of $T$ and $K_0(\{T\}')$}
%%%%%%%%%%%%%%%%%%%%%%%%%%%%%%%%%%%%%%%%%%%%%%%%%%%%%%%%%%
%%%%%%%%%%%%%%%%%%%%%%%%%%%%%%%%%%%%%%%%%%%%%%%%%%%%%%%%%%
%%%%%%%%%%%%%%%%%%%%%%%%%%%%%%%%%%%%%%%%%%%%%%%%%%%%%%%%%%
%%%%%%%%%%%%%%%%%%%%%%%%%%%%%%%%%%%%%%%%%%%%%%%%%%%%%%%%%%

If $T$ is an analytic left invertible operator with $\mbox{\ind}(T)  = -1$, then by Theorem \ref{CD}, $T^* \in B_1(\Omega)$ for $\Omega = \{\lambda: |\lambda|< \epsilon\}$.  Therefore, classifying $\mathcal{S}(T)$ is equivalent to classifying the similarity orbit of Cowen-Douglas operators over a small disc centered at the origin.   The problem of identifying when two Cowen-Douglas operators are similar is a classic one.  In Cowen and Douglas' original work, they show that two operators $R_1, R_2 \in B_1(\Omega)$ are unitarily equivalent if and only if the curvature on the associated hermitian holomorphic vector bundles are equal \cite{Cowen2}.  Cowen and Douglas did not find a similarity classification however.  They asked what is a complete similarity invariant of $B_1(\Omega)$, and more generally, $B_n(\Omega)$. Various authors subsequently worked on this problem, successfully describing the similarity orbit of Cowen-Douglas operators in terms of K-theory.

In \cite{Jiang4}, Jiang describes the similarity orbit of strongly irreducible Cowen-Douglas operators using the $K_0$-group of the commutant algebra.  Later, Jiang, Guo, and Ji gave a similarity classification of all Cowen-Douglas operators using the commutant \cite{Jiang3}. For a summary of these results, and how they connect to the theory of left invertible operators, see \cite{DeSantis}.  The particular result of interest for the Fredholm index $-1$ case is as follows:

\begin{thm}[\cite{Jiang3} - Proposition 5.1.7]
	\label{Jiang_sim}
	Let $A,B \in B_1(\Omega)$.  Then $A$ is similar to $B$ if and only if 
	\[
	K_0(\{A \oplus B\}') \cong \mathbb{Z}.
	\]
\end{thm}

\subsection{Example from Subnormal Operators}
%%%%%%%%%%%%%%%%%%%%%%%%%%%%%%%%%%%%%%%%%%%%%%%%%%%%%%%%%%
%%%%%%%%%%%%%%%%%%%%%%%%%%%%%%%%%%%%%%%%%%%%%%%%%%%%%%%%%%
%%%%%%%%%%%%%%%%%%%%%%%%%%%%%%%%%%%%%%%%%%%%%%%%%%%%%%%%%%
%%%%%%%%%%%%%%%%%%%%%%%%%%%%%%%%%%%%%%%%%%%%%%%%%%%%%%%%%%

We now turn to an important class of non-trivial examples of $\A_T$.   These examples will involve the theory of subnormal operators. We recall the definitions of subnormal operators and minimal normal extensions:

\begin{defn}
	An operator $S \in \BH$ is called \textbf{subnormal} if there exists a Hilbert space $\KK$ such that $\KK\supset \HH$ and a normal operator $N \in \mathscr{B}(\mathscr{K})$ such that 
	
	\begin{enumerate}[i.]
		\item $N \HH \subset \HH$
		\item $S = N \mid_{\HH}$
	\end{enumerate}

	\noindent Such a normal operator $N$ is called a \textbf{normal extension} of $S$.  The operator $N$ is said to be a \textbf{minimal normal extension} if $\KK$ has no proper subspace reducing $N$ and containing $\HH$. 
\end{defn}

Any two minimal normal extensions of a subnormal operator $S$ are unitarily equivalent \cite{Conway2}.  Thus, we usually refer to the minimal normal extension, and denote it by $N:= mne(S)$. 

Classic examples of a subnormal operators are the Toeplitz operators $T_f$ on $H^2(\mathbb{T})$ for $f \in L^\infty(\mathbb{T})$. The minimal normal extension is given by $M_f$ on $L^2(\mathbb{T})$ (for $f$ non-constant).  It is not hard to see that all subnormal operators have this form.   We make the following definition:

\begin{defn}
	Let $S \in \BH$ be a subnormal operator, and $N = mne(S) \in \mathscr{B}(\mathscr{K})$.  If $\mu$ is a scalar-valued spectral measure associated to $N$, and $f \in L^\infty(\sigma(N),\mu)$, we define the \textbf{Toeplitz operator} $T_f \in \BH$ via
	\[
	T_f:= P(f(N)) \mid_{\HH}
	\]
	where $P$ is the orthogonal projection of $\KK$ onto $\HH$. 
\end{defn}
In the case when $S$ is the unilateral shift, the above are the Toeplitz operators on $H^2(\mathbb{T})$.  For any subnormal operator $S$, we have that $T_z = S$, and that $T_{\overline{z}^n} T_{z^m} = T_{\overline{z}^n z^m}$.  Consequently, $\{T_f: f \in C(\sigma(N))\} \subset C^*(S)$. We remark that, while the map from $ L^\infty(\sigma(N),\mu)$ to $\BH$ via $f \mapsto T_f$ is positive and norm decreasing,  it is not multiplicative. 

Ultimately, we are interested in algebras of operators generated by left invertible operators.  Salient examples will arise from the subnormal operators, due in large part to their rich spectral theory.  The following is the first useful result in that direction.

\begin{prop}[\cite{Conway2}]
	\label{normal spectrums}
	Let $S$ be a subnormal operator with $N= mne(S)$.  Then the following inclusions hold:
	\[
	\partial \sigma(S) \subseteq \sigma_{ap}(S) \subseteq \sigma_{ap}(N) = \sigma(N) \subseteq \sigma(S)
	\]
	where $\sigma_{ap}(S)$ is the approximate point spectrum of $S$.
	
\end{prop}

Next we highlight some C*-algebraic results about subnormal operators due to Olin, Thomson, Keough and McGuire. If $N$ is a normal operator, there is a natural identification of $C^*(N)$ with $C(\sigma(N))$ given by the Gelfand transform.  There is also an intimate connection between the C*-algebra generated by a subnormal operator $S$ and its minimal normal extension $N$.  

When $S$ is the unilateral shift, its minimal normal extension $N$ is a unitary.  The commutative C*-algebra $C^*(N) \cong C^*(\sigma(N)) \cong C(\mathbb{T})$ appears in the symbols of the Toeplitz operators. Being a subnormal operator, by definition $S$ dilates to a normal operator.  The unilateral shift also has the additional property the image of $S$ in the Calkin algebra is normal (in fact, unitary). Recall that an operator $S \in \BH$ is called \textit{essentially normal} if its image $\pi(S)$ is normal in the Calkin algebra $\BH/\KH$.

In summary, three key properties that the unilateral shift possesses are irreducibility, sub-normality and essential normality. If $S$ is any operator with these three properties, one obtains a construction similar to the Toeplitz algebra.  It is helpful to view the following theorem with Proposition \ref{normal spectrums} in mind.

\begin{thm}[ \cite{Keough} \cite{McGuire} \cite{Olin} ]
	\label{C-sub}
	If $S$ is an irreducible, subnormal, essentially normal operator, then
	
	\begin{enumerate}[i.]
		\item $\sigma_{ap}(S) = \sigma_e(S)$
		\item For each $f, g \in C(\sigma(N))$, we have
		\begin{enumerate}[a.]
			\item $T_f \in \KH$ if and only if $f$ vanishes on $\sigma_e(S)$
			\item $\|T_f + \KH\| = \|f\|_{\sigma_e(S)}$
			\item $T_{fg} - T_f T_g \in \KH$
			\item $\sigma_e(T_f) = f(\sigma_e(S))$
		\end{enumerate}
		\item Every element of $C^*(S)$ can be written as a sum of a Toeplitz operator and compact:
		\[
		C^*(S) = \{T_f + K : f \in C(\sigma(N)), K \in \KH\}.
		\]
		Moreover, if $\sigma(N) = \sigma_{ap}(S)$, then each element has $A \in C^*(S)$ has a unique representation of the form $T_f + K$.  If $\sigma(N) \neq \sigma_{ap}(S)$, $A$ may be expressed as $A = T_{f_1} + K_1 = T_{f_2} + K_2$, where $f_1\mid_{\sigma_e(S)} = f_2 \mid_{\sigma_e(S)}$.
	\end{enumerate}

\end{thm}

Using the work of Olin, Thomson, Keough and McGuire describing the C*-algebra generated by a subnormal, essentially normal, irreducible operator (Theorem \ref{C-sub}), we characterize the algebras $\A_S$ for $S$ a subnormal, essentially normal left invertible operator. We begin with a simple connection between spectral data of the operators appearing in Theorem \ref{C-sub} and left invertibility.

\begin{lemma}
	\label{sub-left}
	Let $S$ be a subnormal operator with  $N = mne(S)$. If $N$ is invertible, then $S$ is left invertible with $L = T_{z^{-1}}$  a left inverse.  If $\sigma(N) = \sigma_{ap}(S)$, then $S$ is left invertible if and only if $N$ is invertible. 
	
\end{lemma}   

\begin{proof}
	If $N$ is invertible, then  the Toeplitz operator $T_{z^{-1}} = P(N^{-1})\mid_\HH$ is well defined.  Since $N$ is a normal extension of $S$, we have for each $x \in \HH$ 
	\[
	T_{z^{-1}} S x = T_{z^{-1}}(N x) = P(N^{-1} N x) = P x = x.
	\]
	If $\sigma(N) = \sigma_{ap}(S)$, then $S$ is left invertible implies $0 \notin \sigma_e(S) = \sigma(N)$. 		
\end{proof}

Using the basic theory of subnormal operators, we now describe the structure of $\A_S$ for a prototypical class of subnormal operators.

\begin{thmy}
	\label{ess ex}
	Let $S$ be an analytic left invertible, $\mbox{\ind}(S) = -1$, essentially normal, subnormal operator with $N:=mne(S)$ such that $\sigma(N) = \sigma_{ap}(S)$.  Let $\mathscr{B}$ be the uniform algebra generated by the functions $z$ and $z^{-1}$ on $\sigma_e(S)$. Then 
	\[
	\A_S = \{T_f + K: f \in \mathscr{B}, K \in \KH  \}.
	\]
	Moreover, the representation of each element as $T_f + K$ is unique. 
\end{thmy}

\begin{proof}
	By Lemma \ref{sub-left}, $L:= T_{z^{-1}}$ is a left inverse of $S$.  By Corollary \ref{any left}, $\A_S$ is the norm-closed subalgebra of $C^*(S)$ generated by $T_z$ and $T_{z^{-1}}$.  Since $S$ is analytic, it is strongly irreducible, and hence, irreducible.  Therefore by Theorem \ref{C-sub}, each element of $\A_S$ has a unique representation as $T_f + K$ for some $f \in C(\sigma(N))$ and $\sigma(N) = \sigma_{ap}(S) = \sigma_e(S)$. Moreover by Theorem \ref{C-sub}, $L^n = T_{z^{-n}} +K$ for some compact operator $K$.  Since $\A_S$ contains the compacts, it follows that $T_{z^{k}} \in \A_S$ for each $k \in \mathbb{Z}$.  Hence, for each $p \in \mbox{\Alg}(z,z^{-1})$, we have that $T_p \in \A_S$.  Using this information, we now show that $\A_S = \{T_f + K: f \in \mathscr{B}, K \in \KH  \}$.  To do this, it suffices to show that $T_f \in \A_S$ if and only if $f \in \mathscr{B}$. 
	
	First, suppose that $T_f \in \A_S$ for some $f \in C(\sigma(N))$.  Since $\mbox{\Alg}\{T_z, T_{z^{-1}}\}$ is dense in $\A_S$,  for every $\epsilon> 0$ there exists a Laurent polynomial $p \in \mbox{\Alg}(z,z^{-1})$ and compact $K$ such that $\|T_f - (T_p+ K)\| < \epsilon$.  By Theorem \ref{C-sub},
	\[
	\epsilon > \|T_f - (T_p+K)\| = \|T_{f-p} - K\| \geq \|T_{f-p} + \KH\| = \|f - p\|.
	\]
	Hence, $f \in \mathscr{B}$.  For the other inclusion, suppose to the contrary that $f \in \mathscr{B}$ but $T_f \notin \A_S$.  Then there exists a $\delta > 0$ such that for each  $p \in \mbox{\Alg}(z,z^{-1})$ and $K \in \KH$, we have $\|T_f - (T_p + K)\| > \delta$.  In particular, this should hold for any $p$ such that $\|f - p \| < \frac{\delta}{2}$.  Hence
	\[
	\delta \leq \inf_{K \in \KH} \|T_f - (T_p + K)\| = \|T_{f-p} + \KH\| = \|f-p\| < \frac{\delta}{2}
	\]
	which is absurd.  Hence, $T_f$ must be in $\A_S$, completing the proof. 
\end{proof}

Notice that in Theorem \ref{ess ex}, we can drop the requirement that $\sigma(N) = \sigma_{ap}(S)$, so long as the minimal normal extension is invertible. In this case however, one will lose the uniqueness of the representation $T_f + K$ as discussed in Theorem \ref{C-sub}.  As a corollary to Theorem \ref{ess ex}, we get a description of $\A_T$ for analytic Toeplitz operators on $H^2(\mathbb{T})$ with Fredholm index $-1$.

\begin{cor}
	Let $g$ be an analytic function on $\mathbb{T}$ and $X = \ran(g)$ with winding number of $g$ equal to $1$.  Then $\sigma_e(T_g) = X$, and $T_g$ is an analytic left invertible operator with $\mbox{ind}(T_g) = -1$. If $\mathscr{B}$ is the uniform algebra generated by $z$ and $z^{-1}$ on $X$, then we have the following short exact sequence
	\begin{center}
	\begin{tikzpicture}
	\matrix (m) [matrix of math nodes, row sep=3em,
	column sep=3em, text height=1.5ex, text depth=0.25ex]
	{ 0 & \mathscr{K}(H^2(\mathbb{T})) & \A_{T_g} & \mathscr{B} & 0 \\};
	\path[->]
	(m-1-1) edge (m-1-2);
	\path[->]
	(m-1-2) edge node[auto] {$ \iota $} (m-1-3);
	\path[->]
	(m-1-3) edge node[auto] {$ \pi$} (m-1-4);
	\path[->]
	(m-1-4) edge (m-1-5);
	\end{tikzpicture}
\end{center}

Moreover, each element of $\A_{T_g}$ has a unique representation of $T_f + K$ for $f$ in the uniform algebra generated by $g$ and $g^{-1}$ and $K$ compact. 
\end{cor}

The hypotheses of Theorem \ref{ess ex} are natural, but numerous.  This is to guarantee that $S$ remain within our current focus of study. We remark that even if $S$ is left invertible, irreducible, subnormal, essentially normal operator, it need not be analytic.  

Recall, an operator $R \in \BH$ is said to be \textit{cyclic} if there exists an $x \in \HH$ such that $\{R^n x\}_{n=0}^\infty$ is norm dense in $\HH$. A result by Qing shows that every Cowen-Douglas operator is cyclic \cite{Qing}. While all Cowen-Douglas operators must be cyclic, the adjoints of general subnormal operators need not be cyclic.  A long-standing problem posed by Deddens and Wogen asked which subnormal operators had cyclic adjoints \cite{Conway2}. Feldman answered this question in \cite{FeldmanN}.  A subnormal operator is said to be \textit{pure} if it has no non-trivial normal summand.  Every subnormal operator can be decomposed as $S = S_p \oplus N$, where $S_p$ is pure and $N$ is normal.  The general cyclicity result is as follows:

\begin{thm}[Feldman \cite{FeldmanN}]
	\label{Feldman}
	If $S = S_p \oplus N$ is a subnormal operator, then $S^*$ is cyclic if and only if $N$ is cyclic.  In particular, pure subnormal operators have cyclic adjoints. 
\end{thm}

Having a cyclic vector clearly is not sufficient for an operator to be Cowen-Douglas.  However, Theorem \ref{Feldman} is a condition of necessity. Thomson showed in \cite{Thomson} that if $S$ is a pure, cyclic subnormal operator, then $S^*$ is Cowen-Douglas.  However, as far the author is aware, there is no known elementary equivalence to guarantee $S^*$ is Cowen-Douglas.

We remark that the similarity orbit of subnormal operators was classified by Conway \cite{Conway3}.  He showed two subnormal operators are similar if and only if the scalar valued spectral measure associated to the minimal normal extensions were the same. In this case, there is no need to investigate the $K_0$ group of the commutant.  Rather, the spectral data encodes all the information about the similarity orbit.

\section*{Acknowledgements}

The author would like to thank David Pitts for his mathematical insights and intuition.

\appendix

%% Appendices go here (if you have them)

%%%%%%%%%%%%%%%%%%%%%%%%%%%%%%%%%%%%%%%%%
%%%%%%%%%%%%%%%%%%%%%%%%%%%%%%%%%%%%%%%%%
%%%%%%%%%%%%%%%%%%%%%%%%%%%%%%%%%%%%%%%%%
%%%%%%%%%%%%%%%%%%%%%%%%%%%%%%%%%%%%%%%%%
\bibliographystyle{plain}
%\bibliography{References}

%% BibTeX is your friend
%% Bibliography goes here (You better have one)
%% BibTeX is your friend
%\bibliographystyle{unsrt}
\bibliography{biblo}
%\nocite{*}
%%%%%%%%%%%%%%%%%%%%%%%%%%%%%%%%%%%%%%%%%
%%%%%%%%%%%%%%%%%%%%%%%%%%%%%%%%%%%%%%%%%
%%%%%%%%%%%%%%%%%%%%%%%%%%%%%%%%%%%%%%%%%
%%%%%%%%%%%%%%%%%%%%%%%%%%%%%%%%%%%%%%%%%

%% Index go here (if you have one)
\end{document}